\documentclass[11pt]{amsart}
\usepackage{amsfonts}
\usepackage{amsmath,graphicx,amssymb,amsthm}
\usepackage{amsmath}
\usepackage{amssymb}
\usepackage{graphicx}
\usepackage{comment}
\usepackage[backref]{hyperref}
\providecommand{\U}[1]{\protect\rule{.1in}{.1in}}

\usepackage[all]{xypic} 
\usepackage{tikz-cd}

\def\U{\mathcal U}

\def\amslatex{$\mathcal{A}\kern-.1667em\lower.5ex\hbox{$M$}\kern-.125em\mathcal{S}$-\LaTeX}

\newcommand{\abs}[1]{\left\vert#1\right\vert}
\newcommand{\norm}[1]{\left\Vert#1\right\Vert}

\newtheorem{set}{set}[section]

\newcommand{\enifed}{\mathrel{\hbox{$\equiv$\hskip -.90em \lower .47ex \hbox{$\rightharpoondown$}}}}

\newtheorem{theorem}[set]{Theorem}
\theoremstyle{plain}

\newtheorem{corollary}[set]{Corollary}

\newtheorem{definition}[set]{Definition}

\newtheorem{lemma}[set]{Lemma}

\newtheorem{proposition}[set]{Proposition}
\newtheorem{remark}[set]{Remark}

\allowdisplaybreaks[4]

\begin{document}
\title{On Noncommutative Joinings}


\author{Jon Bannon}
\address{Siena College Department of Mathematics, 515 Loudon Road, Loudonville,
NY\ 12211, USA}
\email{jbannon@siena.edu}
\author{Jan Cameron}
\address{Department of Mathematics and Statistics, Vassar College, Poughkeepsie, NY 12604, USA}
\email{jacameron@vassar.edu}
\author{Kunal Mukherjee}
\address[3]{Indian Institute of Technology Madras, Chennai 600 036, India}
\email{kunal@iitm.ac.in}

\begin{abstract}
This paper extends the classical theory of joinings of measurable dynamical systems to the noncommutative setting from several interconnected points of view. Among these is a particularly fruitful identification of joinings with equivariant quantum channels between $W^{\ast}$--dynamical systems that provides noncommutative generalizations of many fundamental results of classical joining theory. We obtain fully general analogues of the main classical disjointness characterizations of ergodicity, primeness and mixing phenomena. 
\end{abstract}
\maketitle
\section{Introduction}
Ergodicity and mixing properties of group actions play a central role in the structure theory of von Neumann algebras, unifying the study of dynamics in the commutative and noncommutative settings, while encoding important algebraic and analytical information about von Neumann algebras.  The notion of a joining of two dynamical systems has proved to be fundamental to the study of these phenomena in both the commutative and noncommutative contexts. The earliest appearance of the joining concept in the noncommutative context is recognized as Markov maps in the notion of \textit{stationary coupling} in \cite{SaTh}. The more recent literature contains several papers developing a theory of joinings in the noncommutative setting understood as invariant states on algebraic tensor products \cite{Du,Du2,Du3}. In the above papers, joinings of a system with its ``mirror image" were considered (cf. Def. 2.2 of \cite{Du2}), as well as joinings with modular invariant systems (cf. Def. 2.3 of \cite{Du3}), for technical reasons. The present paper develops a point of view on noncommutative joinings that brings together those found in \cite{SaTh} and \cite{Du,Du2,Du3} in a way that regards the appearance of the mirror image and modular invariance as crucial. In the context of finite von Neumann algebras, the point of view we adopt was first developed in [KLP], where the associated disjointness was employed to establish generic turbulence results for group actions on the hyperfinite $\rm{II}_{1}$ factor. Extending the point of view taken in [KLP], in this paper we develop a general theory of joinings for noncommutative dynamical systems arising from state-preserving actions on von Neumann algebras.

H. Furstenberg \cite{Fu} developed what can be viewed as an arithmetic of dynamical systems, in which systems are analogous to natural numbers, and subsystems\footnote{Since this paper primarily aims to transport ideas from ergodic theory to noncommutative operator algebras, we reserve the word \textit{factor} to mean von Neumann algebra with trivial center.} are analogous to divisors. Furstenberg calls two systems \textit{relatively prime} if their only common subsystem is trivial, and \textit{disjoint} if whenever the two systems appear as subsystems of a third, their product also appears as a subsystem. Furstenberg asks whether these are equivalent, deeming disjointness the more useful notion and applying it to classify dynamical systems. D. Rudolph's negative answer in the measure--theoretic setting to Furstenberg's question marks the birth of the theory of joinings, his construction naturally providing counterexamples to many questions in ergodic theory \cite{Ru}. The theory has since been enriched by the work of many authors, e.g. \cite{Fu2}, \cite{Ve}, \cite{Ra}, \cite{JR}, \cite{Ho}, \cite{Ki}, \cite{GlHoRu}, \cite{Rh}, \cite{JLM} and \cite{LPT}. We refer the reader to \cite{Gl} for an overview of the classical theory of joinings.

This paper continues the study of noncommutative joinings in the spirit of \cite{KLP}, defining joinings of $W^{*}$--dynamical systems as quantum channels that simultaneously intertwine the group actions and modular time evolutions of the systems. This extends the definition of joining in \cite{KLP} and, as we will show, leads to sharp noncommutative analogues of a number of landmark classical results in ergodic theory. Our approach overcomes certain technical complications associated with correspondences that arise in the setting of infinite von Neumann algebras (see, for example, \cite{Co} and \cite{Fa}), by leveraging the symmetry of the standard form. Furthermore, the completely positive maps associated to joinings in the present sense are shown to form a subclass of the Markov maps originally considered in [AD], namely the subclass of equivariant Markov maps. We thus take a dynamical perspective on coherent equivariant quantum channels and a quantum information--theoretic interpretation of joinings, reminiscent of results found in the recent physics literature \cite{KuSc, Ma, Sc}. Before discussing specifics of our results, we briefly place them in physical context.

The appearance of \textit{equivariance} in the connection between joinings and quantum channels established here marks relationships with basic physics. Quantum channels that are equivariant with respect to a global symmetry group play a role in the statistical mechanics of open quantum systems \cite{Hol1}, \cite{Hol2}, \cite{Hol3}. Furthermore, in quantum information theory, equivariance of a quantum channel facilitates the computation of its classical capacity \cite{GHGP}. Often in quantum field theory a scattering operator (S-matrix) is a unitary intertwiner of a representation of the Poincar\'e group that fixes the vacuum, and implements a Poincar\'e equivariant scattering automorphism of the algebra of bounded operators on the Hilbert space of the interaction. This provides an example of an equivariant quantum channel that encodes a lot of information about a quantum field theory (cf. pp. 349--350 of \cite{Se}). In fact, in \cite{Le1}, \cite{Le2} a low--dimensional model of an asymptotically complete local quantum field theory is recovered entirely from its two--particle S--matrix, indicating that under certain circumstances the S--matrix encodes enough data to reconstruct a vacuum representation of the local net of type $\rm{III}_{1}$ factors of a quantum field theory.

We note that the Markov maps introduced in \cite{AD} (in finite dimensions and without equivariance) were employed in \cite{HaMu} to settle the Asymptotic Quantum Birkhoff Conjecture, renewing interest in operator methods in quantum information theory while establishing new connections of Markov maps with the Connes Embedding Conjecture in operator algebras. Also in connection with quantum information theory, Tsirelson's problem of whether the ``external'' and ``internal'' models of a pair of compatible quantum systems give rise to the same set of correlations has recently been shown to be equivalent to the Connes Embedding Conjecture \cite{JNPPSW}, \cite{Oz}. The ``external'' model is germane to the nonrelativistic quantum mechanics of open systems and the ``internal'' model arises naturally in local quantum field theory. In the latter setting, a physically reasonable nuclearity assumption (the split property) implies that every local algebra of observables is the separably acting hyperfinite type $\rm{III}_1$ factor, consequently controlling the number of local degrees of freedom enough to yield an affirmative answer to Tsirelson's Problem for causally complementary local algebras of observables \cite{BAF}. Despite this, in a vacuum representation of a local net where the Reeh--Schlieder Theorem holds,  there is a dense set of entangled states on the von Neumann algebra generated by the local algebras of any two causally complementary regions of spacetime, even if one considers non--interacting free quantum fields \cite{HaCl}. Furthermore, an arbitrary entangled state in such a theory maximally violates the BCHS inequality \cite{SuWe}, suggesting that entangled states on $\rm{III}_1$ factors arising from pairs of causally complementary local observable algebras in algebraic quantum field theory provide extreme test cases of Tsirelson's Problem. If a local net of observable algebras additionally satisfies Haag duality, then the restriction of any state from the algebra of quasilocal observables to the $\ast$--subalgebra generated by a local observable algebra and its commutant yields a joining in our sense. The study of joinings in the type $\rm{III}$ setting thus may be viewed as aiming to refine our understanding of quantum correlations. 

We now describe the results of the present paper.

A state on the algebraic tensor product of one von Neumann algebra with the opposite of another yields a GNS representation that is naturally regarded as a cyclic bimodule of the two von Neumann algebras. We note that at the $C^{*}$-algebra level such states play an important role in Woronowicz's study of purification of factor states in \cite{Wo}. If the marginals (restrictions to the tensor factors) of the state are both normal, then the resulting bimodule is binormal, i.e. the bimodule is a \textit{correspondence}. If, in addition, a separable locally compact group acts on each of the von Neumann algebras in a way that preserves each of the marginals, then these group actions intertwine their respective modular automorphism groups. In analogy with the classical case, if the state preserves the doubled action of the group on the tensor product (respecting the opposite) as well as the analogously doubled modular group action, we regard the state as a joining of the two dynamical systems of the tensor factors. The need for the opposite algebra, and hence the bimodule point of view, is fundamental for the results we prove in this paper, even though it seems like only slight formal shift from the classical definition. For the details of the above discussion, see Theorem \ref{BimoduleFromJoining}. This additional joining data on the state is encoded as extra structure on the above correspondence that is manifested in properties of the distinguished cyclic vector: This vector generates copies of the standard left (and respectively right) modules of the two von Neumann algebras and is fixed by unitary representations induced by the (doubled) group and modular actions (Theorem \ref{JoiningFromBimodule}). Transporting the extra structure to the completely positive map associated to the cyclic correspondence reveals that joinings are precisely the equivariant maps that are Markov with respect to the marginal states (Theorems \ref{Popa's Lemma copy(1)} and \ref{CharacterizeJoins}), and hence are precisely the equivariant u.c.p. maps that admit adjoints (Theorem \ref{AdjointisEquivariant} and preceding remarks).

The interpretation of joining as equivariant Markov map leads to natural proofs that ergodic systems are precisely those that are disjoint from all trivial systems (Theorem \ref{Ergodic}) and that prime systems (systems with only trivial subsystems) are those whose nonidentity self--joinings fix only scalars (Theorem \ref{Berg}). We provide a direct analogue of the classical topological characterization of mixing in terms of joinings (Theorem \ref{Mixing}). More substantially, though, our perspective on joinings provides the first full joining characterization of weak mixing in the noncommutative setting, something that has eluded capture in other recent attempts to develop a noncommutative joining theory \cite{Du}, \cite{Du2}, \cite{Du3}.  Classically, a system is weakly mixing if and only if it is disjoint from all compact systems, i.e. those systems whose Koopman representations admit only precompact orbits. In our setting, an equivariant Markov map always induces an intertwiner of the analogue of the Koopman representation that preserves such precompactness and hence yields a finite--dimensional subrepresentation, and consequently weakly mixing systems must be disjoint from all compact systems (Theorem \ref{WmDisjointCompact}).

The fact that disjointness from all compact systems implies that a system is weakly mixing (Theorem \ref{CharacterizeWeakMixing}) is a result concerning rigidity in the Ergodic Hierarchy. Such results have an interesting history apart from joining theory. The first such rigidity result is due to Halmos \cite{Hal}: Any ergodic automorphism of a compact abelian group is mixing, and its spectral measure is Lebesgue with infinite multiplicity. Another such result was proved by Jak\u{s}i\'{c} and Pillet \cite{JP}: An ergodic action of $\mathbb{R}$ on any von Neumann algebra preserving a faithful normal state is automatically weakly mixing. (Particularly, this implies that the action of the modular automorphism group of a faithful normal state on a von Neumann algebra is weakly mixing whenever it is ergodic.) In the present paper we obtain a rigidity result for finite--dimensional invariant subspaces of the canonical unitary representation on the GNS Hilbert space that is induced by a state--preserving ergodic action of a given separable locally compact group. This rigidity is rooted in the symmetry of the standard form and the respective uniqueness of the polar decomposition and spectral decomposition of affiliated operators. In Theorem \ref{eigencase}, we show that any eigenvector orthogonal to the vacuum in the GNS Hilbert space of an ergodic action comes from a trace--zero unitary in the centralizer of the state, and then establish in Theorem \ref{InvaraintSubspaceInsideCentralizer} that every finite--dimensional invariant subspace of the GNS Hilbert space under the hypothesis of ergodicity must lie in the image of this centralizer.  

The proof of Theorem \ref{OtherwayWmDisjointCompact} employs the above rigidity to manufacture a joining of an ergodic system whose induced unitary representation has a finite--dimensional invariant subspace with an ergodic system whose underlying von Neumann algebra is either finite--dimensional or abelian. Thus, under the hypothesis of ergodicity, we can locate a large region of randomness of the action. We point out that although most physical systems are not ergodic, almost all (in a Baire category sense) Hamiltonian systems are not integrable, and consequently exhibit randomness on large parts of their phase space. It is therefore interesting physically to be able to locate such chaotic parts of a dynamical system.

The layout of the paper is as follows. In \S2 we provide preliminary background. In \S3 we define joining as a state and characterize joinings as pointed correspondences and then in \S4 characterize joinings as equivariant Markov maps. Basic examples appear in \S5. Our results on ergodicity and mixing properties appear in \S6. Consequences of the weak mixing results are discussed in \S7.

\vspace{5mm}
\noindent\textbf{Acknowledgements:}
The authors thank David Kerr for suggesting that we investigate the basic connection between bimodules and joinings. Work on this paper was initiated during a visit of KM to Vassar College and Siena College in 2012, partially supported by Vassar's Rogol Distinguished Visitor program.  We thank the Rogol Fund for this support.  JC's research was partially supported by a research travel grant from the Simons Foundation, and by Simons Foundation Collaboration Grant for Mathematicians \#319001. JB thanks Liming Ge and the Academy of Mathematics and Systems Science of the Chinese Academy of Sciences for their hospitality and support. KM thanks Serban \c{S}tr\v{a}til\v{a} for helpful conversations. The authors thank Ken Dykema, David Kerr, Vern Paulsen, and an anonymous referee for helpful comments on earlier versions of this paper. Finally, the authors thank the anonymous referees of this paper for their careful reading of the paper and subsequent extraordinarily helpful suggestions, including an alternative approach to Theorem \ref{InvaraintSubspaceInsideCentralizer}.    


\section{Preliminaries}

Let $M$ be a von Neumann algebra with a faithful normal state $\varphi$. Denote by $L^{2}(M,\varphi)$ and $\Omega_{\varphi}$ the associated GNS Hilbert space and its canonical unit cyclic and separating vector.

\subsection{Tomita-Takesaki Theory}
We recall, without proof, the following facts that are needed in the sequel. The conjugate--linear map defined by $S_{0,\varphi}x\Omega_{\varphi}=x^{\ast}\Omega_{\varphi}$ for all $x\in M$ is closable with closure $S_{\varphi}$ having polar decomposition $J_{\varphi}\Delta_{\varphi}^{1/2}$. In fact, the adjoint $S_{\varphi}^{\ast}$ is the closure of the closable linear map on $L^{2}(M,\varphi)$ defined by $F_{0,\varphi}x^{\prime}\Omega_{\varphi}=(x^{\prime})^{\ast}\Omega_{\varphi}$ for $x^{\prime}\in M^{\prime}$,
and the polar decomposition of $S_{\varphi}^{\ast}$ is $J_{\varphi}%
\Delta_{\varphi}^{-1/2}$. The conjugate--linear map $J_{\varphi}:L^{2}(M,\varphi) \rightarrow L^{2}(M,\varphi)$ satisfies $J_{\varphi}^{2}=1$ and $\langle J_{\varphi}\xi,J_{\varphi}\eta\rangle=\langle\eta,\xi\rangle$ for all $\xi,\eta\in L^{2}(M,\varphi)$. It follows that $J_{\varphi}=J_{\varphi}^{\ast}$ as a conjugate--linear map. Tomita's modular operator is the positive self--adjoint operator $\Delta_{\varphi}=S_{\varphi}^{\ast}S_{\varphi}$. The operator $\Delta_{\varphi}$ is invertible and satisfies
$J_{\varphi}\Delta_{\varphi}J_{\varphi}=\Delta_{\varphi}^{-1}$, as well as $S_{\varphi}=J_{\varphi}\Delta_{\varphi
}^{1/2}=\Delta_{\varphi}^{-1/2}J_{\varphi}$ and $S_{\varphi}^{\ast}%
=J_{\varphi}\Delta_{\varphi}^{-1/2}=\Delta_{\varphi}^{1/2}J_{\varphi}$. Furthermore, $\Delta_{\varphi}^{it}J_{\varphi}=J_{\varphi}\Delta_{\varphi}^{it}$ and $\Delta_{\varphi}^{it}\Omega_{\varphi}=J_{\varphi}\Omega_{\varphi}=\Omega_{\varphi}$ for all $t\in\mathbb{R}$. By the fundamental theorem of Tomita and Takesaki, $\Delta_{\varphi}^{it}M\Delta_{\varphi}^{-it}=M$
for all $t\mathbb{\in R}$, and $J_{\varphi}MJ_{\varphi}=M^{\prime}$. Recall that $\sigma_{t}^{\varphi}(x)=\Delta_{\varphi}^{it}x\Delta_{\varphi}^{-it}$ for all $x\in M$ defines the modular automorphism of $M$ associated to $t \in \mathbb{R}$. For more detail we refer the reader to \cite{Str}, \cite{Ta2}.

In this paper we frequently work with the centralizer $$M^{\varphi}=\{x\in M:\varphi(xy)=\varphi(yx)\text{ } \forall y\in M\}$$ of the state $\varphi$. The centralizer $M^{\varphi}$ coincides with the set of fixed points of the modular automorphism group $\{\sigma_{t}^{\varphi}\}_{t\in \mathbb{R}}$. If $M$ is a type $\rm{III}_{\lambda}$ factor with $\lambda \neq 1$ then $M^{\varphi}$ contains a maximal abelian $*-$ subalgebra of $M$. In contrast, there exist type $\rm{III}_{1}$ factors $M$ such that $M^{\varphi}=\mathbb{C}1$.

\subsection{The Opposite Algebra of $M$}

What follows requires the opposite algebra $M^{op}=M^{\prime}\cap\mathbf{B}(L^{2}(M,\varphi))$. The
map $x\mapsto x^{op}=J_{\varphi}x^{\ast}J_{\varphi}$ is a $\ast$%
--antiisomorphism of $M$ with $M^{op}$. By the double commutant theorem
$(M^{op})^{op}=M$. The state $\varphi$ on $M$ is the restriction to $M$ of the
vector state $\omega_{\Omega_{\varphi}}$ of $\mathbf{B}(L^{2}(M,\varphi))$.
Let $\varphi^{op}$ denote the restriction of $\omega_{\Omega_{\varphi}}$ to
$M^{op}$. It follows that $\varphi^{op}(x^{op})=\varphi(x)$ for all $x\in M$.
Let $ \theta \in Aut(M,\varphi)$, i.e. $ \theta $ is an automorphism of $M$ such that $\varphi\circ \theta =\varphi$. 

Since $M$
is in standard form, the automorphism $ \theta $ has a canonical implementation by a unitary
$U\in\mathbf{B}(L^{2}(M,\varphi))$ that fixes $\Omega_{\varphi}$, namely the unitary extending $U(x\Omega_{\varphi})=\theta(x)\Omega_{\varphi}$ for $x\in M$. In this paper this choice of unitary implementing a state-preserving automorphism will be implicitly made whenever possible.  It is a basic fact that if $UMU^{\ast
}=M$ then $UM^{op}U^{\ast}=M^{op}$. We hence define the opposite automorphism
$ \theta ^{op}$ of $M^{op}$ as $ \theta ^{op}(x^{op})=Ux^{op}U^{\ast}$ for all
$x\in M$. Then $ \theta ^{op}\in Aut(M^{op},\varphi^{op})$. It is clear that $S_{\varphi}=US_{\varphi}U^{*}$ and hence by polar decomposition $J_{\varphi}=UJ_{\varphi}U^{*}$ and $\Delta^{1/2}_{\varphi}=U\Delta^{1/2}_{\varphi}U^{*}$, from which it readily follows that $\theta(x)=(\theta^{op}(x^{op}))^{op}$ for all $x \in M$. If $\beta: G \rightarrow Aut(M,\varphi)$ is a group action, then the
opposite action $\beta^{op}: G \rightarrow Aut(M^{op},\varphi^{op})$ is defined by $\beta_{g}^{op}(y^{op})=(\beta_{g}(y))^{op}$ for each $g\in G$ and $y\in M$.

\subsection{Pointed Correspondences}

Let $N$ and $M$ be von Neumann algebras. A pointed $N$--$M$ correspondence
$(\mathcal{H},\xi)$ is a binormal $N$--$M$ Hilbert bimodule $\mathcal{H}=$
$_{N}\mathcal{H}_{M}$ with a distinguished unit cyclic vector $\xi$. Every
pointed correspondence can be obtained by the Stinespring construction applied
to a unital completely positive map $\Phi:N\rightarrow M$.
In \S3 and \S4 below we develop the necessary correspondence theory for von Neumann algebras of arbitrary type, with respect to given faithful normal states in keeping with the ergodic theory picture. To better facilitate the future crossover of techniques from the study of finite von Neumann algebras, we have worked to stay close to the contemporary notation of \cite{Po} in what follows. A benefit of this choice is that it highlights the subtle technical differences that must be overcome in order to handle factors of arbitrary type. 

We note that an $N$--$M$ correspondence alternatively may be viewed as a pair of commuting $*-$representations $\pi_{N}$ and $\pi_{M^{op}}$ that encode, respectively, the left $N$ and right $M$ actions in the bimodule picture. This notation will be used throughout the paper to denote these actions wherever confusion is possible.

\subsection{$W^{\ast}$--Dynamical Systems}

A $W^{\ast}$--dynamical system is denoted by $\mathfrak{N}=(N,\rho,\alpha, G )$, with $N$ a von Neumann algebra, $\rho$ a faithful
normal state on $N$ and $\alpha$ a strongly continuous action of a separable
locally compact group $ G $ on $N$ by $\rho$--preserving automorphisms. In
the sequel let $\mathfrak{N}=(N,\rho,\alpha, G )$ and $\mathfrak{M}
=(M,\varphi,\beta, G )$ be $W^{\ast}$--dynamical systems. We remark that although $G$ is assumed separable and locally compact, many of the results below are valid for a much larger class of groups. 

Since $ G $
acts by state--preserving automorphisms,
for each $g\in G $ the corresponding automorphism $\alpha_{g}$ $($resp.
$\beta_{g})$ is implemented on $L^{2}(N,\rho)$ $($resp. $L^{2}(M,\varphi))$ by
a unitary $U_{g}$ $($resp. $V_{g})$. It is easy to check that $\pi_{U}:g\mapsto U_{g}$
$($resp. $\pi_V:g\mapsto V_{g})$ is a strongly continuous unitary representation of $ G $
on $L^{2}(N,\rho)$ $($resp. $L^{2}(M,\varphi))$. Similarly when
the action of $ G $ is replaced by the action of $\mathbb{R}$ by modular automorphisms the associated unitary representation on $L^{2}(N,\rho)$
$($resp. $L^{2}(M,\varphi))$ is $\pi_{{\Delta}_{\rho}}:t\mapsto\Delta_{\rho}^{it}$ $($resp.
$\pi_{{\Delta}_{\varphi}}:t\mapsto\Delta_{\varphi}^{it})$ for $t\in\mathbb{R}$. 

If $\alpha$ is an automorphism of $M$ then $\sigma_{t}^{\varphi\circ\alpha}=\alpha^{-1}\circ\sigma_{t}^{\varphi}\circ\alpha$ for all $t\in\mathbb{R}$ and consequently $\varphi\circ\alpha=\varphi$ implies that $\alpha$ commutes with every modular automorphism $\sigma_{t}^{\varphi}$ (cf. Theorem 1 of \cite{HeTa}).

\subsection{Notation}

If there is danger of confusion, we sometimes label the inner product and norm
of a Hilbert space $\mathcal{H}$ as $\langle\cdot,\cdot\rangle_{\mathcal{H}}$
and $\norm{\cdot}_{\mathcal{H}}$. For Hilbert spaces obtained via the GNS (resp.
Stinespring) construction of a state $\varphi$ (resp. completely positive map
$\Phi$) we will denote the inner product and norm by $\langle\cdot
,\cdot\rangle_{\varphi}$ and $\norm{\cdot}_{\varphi}$ (resp. $\langle\cdot
,\cdot\rangle_{\Phi}$ and $\norm{\cdot}_{\Phi}$). There are several uses of `$J$' in the paper, which will be distinguished using subscripts. The reader should beware. Finally, for a set of operators $\mathcal{S}$, let $vN (\mathcal{S})$ denote the von Neumann algebra $(\mathcal{S} \cup \mathcal{S}^{*})''$ in what follows.


\section{Joinings and Correspondences}

The notion of doubling in von Neumann algebras as manifested in the theory
correspondences is a conceptual breakthrough, due to Connes, that is immensely
valuable and is at the heart of the subject. We introduce the following notion of joining for $W^{\ast}$--dynamical systems that incorporates the time evolution of the systems and is compatible with the theory of correspondences.

We have attempted to present the following foundational results with minimal technical artifice. Many of these results admit simpler proofs when all of the von Neumann algebras in question are finite. The passage from the finite to infinite setting is largely possible because of the symmetry properties of the standard form, and requires the usual care needed if one is working in the non-tracial setting.   

We point out that the modular invariance requirement in Eq. \eqref{Eq: Diagonal Action Preservation} below may now seem perfunctory, but it is critical for obtaining fully noncommutative analogues of the classical joining characterizations of mixing properties. This is evident from \S6 onward.   

\begin{definition}
\label{DefinitionJoiningasState}A \textbf{joining} of the $W^{\ast}
$--dynamical systems $\mathfrak{N}$ and $\mathfrak{M}$ is a state $\omega$ on
the algebraic tensor product $N\odot M^{op}$ satisfying
\begin{align}
\omega(x\otimes1_{M}^{op})  &  =\rho(x),\label{Eq: Pushforwards}\\
\omega(1_{N}\otimes y^{op})  &  =\varphi(y),\nonumber
\end{align}
and%
\begin{align}
\omega\circ(\alpha_{g}\otimes\beta_{g}^{op})  &  =\omega,
\label{Eq: Diagonal Action Preservation}\\
\omega\circ(\sigma_{t}^{\rho}\otimes(\sigma_{t}^{\varphi})^{op})  &
=\omega,\nonumber
\end{align}
for all $x\in N$, $y\in M$, $g\in G $ and $t\in\mathbb{R}$. We denote by
$J_{s}(\mathfrak{N},\mathfrak{M})$ the set of all joinings \footnote{Here the subscript `s' means `state'.} of
$\mathfrak{N}$ and $\mathfrak{M}$.
\end{definition}

Note that $M=M^{op}$ when $M$ is abelian and the modular theory of $M$ is
trivial when $M$ is finite, thus the modification Eq. \eqref{Eq: Diagonal Action Preservation} required for properly
infinite algebras is masked in the classical case. Despite its apparent subtlety the above definition will allow us to reinterpret joinings as time covariant coherent symmetric quantum channels in \S4, revealing a vivid analogy between ergodic theory and quantum information theory.

Given $\omega\in J_{s}(\mathfrak{N},\mathfrak{M})$ we form the GNS Hilbert
space $\mathcal{H}_{\omega}$ that is the separation and completion of $N\odot
M^{op}$ with respect to the inner product extending
$\langle x_{0}\otimes y_{0}^{op},x_{1}\otimes y_{1}^{op}\rangle:=\omega
(x_{1}^{\ast}x_{0}\otimes(y_{0}y_{1}^{\ast})^{op})$ for all $x_{0},x_{1}\in N$
and $y_{0},y_{1}\in M$.

We will repeatedly use the following \textquotedblleft exchange
maps\textquotedblright, which switch the distinguished cyclic vectors in
pointed Hilbert modules. 

\begin{definition}
\label{ExchangeMap}Let $N$ be a von Neumann algebra and let $(\mathcal{H},\xi)$ and $(\mathcal{K},\eta)$ be cyclic normal left (resp. right) $N$--modules, with
$\xi$ and $\eta$ separating vectors. The \textbf{left (resp. right) exchange
map} is the densely defined linear map $L_{\xi,\eta}:\mathcal{H}%
\mapsto\mathcal{K}$ (resp. $R_{\xi,\eta}$) with domain $N\xi$ (resp. $\xi N$)
defined by $L_{\xi,\eta}(x\xi)=x\eta$ (resp. $R_{\xi,\eta}(\xi x)=\eta x$) for
all $x\in N$.
\end{definition}

We now prove that if the GNS construction is applied to a joining $\omega$,
then the pair $(\mathcal{H}_{\omega},\xi_{\omega})$ formed from the resulting
Hilbert space $\mathcal{H}_{\omega}$ and unit cyclic vector $\xi_{\omega}$ is
a pointed $N$--$M$ correspondence\footnote{Note that in our notation for
canonical GNS\ cyclic vectors, we reserve the use of $\Omega_{\rho}$ to
signify the cyclic vector of a faithful normal state $\rho$ on a von Neumann
algebra, whereas $\xi_{\omega}$ will be used to denote the cyclic vector
obtained from applying the GNS construction to a joining $\omega$, in order to
remind the reader that a joining is not a state on a von Neumann algebra.}
with extra structure coming from Eq. \eqref{Eq: Pushforwards} and
\eqref{Eq: Diagonal Action Preservation}.

\begin{theorem}
\label{BimoduleFromJoining}The GNS Hilbert space $\mathcal{H}_{\omega}$
obtained from $\omega\in J_{s}(\mathfrak{N},\mathfrak{M})$ is a pointed $N$--$M$ correspondence with cyclic vector $\xi_{\omega}$ given by
the class of $1_{N}\otimes1_{M^{op}}$ in $\mathcal{H}_{\omega}$, and bimodule structure obtained from the GNS representation $\pi_{\omega}$ as $x\eta y:=\pi_{\omega}(x\otimes y^{op})\eta$ for all $x\in N$, $y\in M$ and $\eta\in\mathcal{H}_{\omega}$. Let
$\overline{N\xi_{\omega}}$ (resp. $\overline{\xi_{\omega}M}$) denote the
closure of $N\xi_{\omega}$ (resp. $\overline{\xi_{\omega}M}$) in
$\mathcal{H}_{\omega}$. Then the exchange map $L_{\xi_{\omega},\Omega_{\rho}
}:\overline{N\xi_{\omega}}\rightarrow L^{2}(N,\rho)$ (resp. $R_{\xi_{\omega
},\Omega_{\varphi}}:\overline{\xi_{\omega}M}\rightarrow L^{2}(M,\varphi)$)
extends to a normal left (resp. right) Hilbert module isomorphism of
$\overline{N\xi_{\omega}}$ (resp. $\overline{\xi_{\omega}M}$) onto
$L^{2}(N,\rho)$ (resp. $L^{2}(M,\varphi)$). Furthermore, for each $g\in G $
and $t\in\mathbb{R}$
\begin{equation}
x\xi_{\omega}y\overset{\Pi_{g}}{\mapsto}\alpha
_{g}(x)\xi_{\omega} \beta_{g}(y)
\label{Eq: UnitaryRepGammaonHomega}
\end{equation}
and
\begin{equation*}
x\xi_{\omega}y\overset{T_{t}}{\mapsto}\sigma_{t}
^{\rho}(x)\xi_{\omega}\sigma_{t}^{\varphi}(y)
\end{equation*}
extend to unitary operators on $\mathcal{H}_{\omega}$. The maps $g \mapsto \Pi_{g}$ and $t\mapsto T_{t}$ are strongly continuous unitary representations of $ G $ and $\mathbb{R}$ on $\mathcal{H}_{\omega}$.
\end{theorem}

\begin{proof}
By Eq. \eqref{Eq: Pushforwards} $\xi_{\omega}$ is separating for the left (resp. right) action of $N$ (resp. $M$) on $N\xi_{\omega}$  (resp. $\xi_{\omega}M$). Consider the left
exchange map $L_{\xi_{\omega},\Omega_{\rho}}$ from $\overline{N\xi_{\omega}}$ into
$L^{2}(N,\rho)$. This is clearly a left $N$--module map. We have for all
$x,y\in N$ that
\begin{equation}
\langle x\xi_{\omega},y\xi_{\omega}\rangle_{\omega}=\omega(y^{\ast}
x\otimes1_{M^{op}})=\rho(y^{\ast}x)=\langle x\Omega_{\rho},y\Omega_{\rho
}\rangle_{\rho} \label{Eq: ModuleIsometry}
\end{equation}
and hence $L_{\xi_{\omega},\Omega_{\rho}}$ extends to a unitary.

The right exchange map $R_{\xi_{\omega},\Omega_{\varphi}}$ is clearly a right $M$--module
map. We have for all $x,y\in M$
\begin{align*}
\langle\xi_{\omega}x,\xi_{\omega}y\rangle_{\mathcal{\omega}}  &  =\omega
(1_{N}\otimes(xy^{\ast})^{op})=\varphi(xy^{\ast})\\
&  =\varphi((x^{\ast})^{\ast}y^{\ast})=\langle y^{\ast}\Omega_{\varphi
},x^{\ast}\Omega_{\varphi}\rangle_{\varphi}=\langle J_{\varphi}x^{\ast}
\Omega_{\varphi},J_{\varphi}y^{\ast}\Omega_{\varphi}\rangle_{\varphi}\\
&  =\langle J_{\varphi}x^{\ast}J_{\varphi}\Omega_{\varphi},J_{\varphi}y^{\ast
}J_{\varphi}\Omega_{\varphi}\rangle_{\varphi}=\langle\Omega_{\varphi}
x,\Omega_{\varphi}y\rangle_{\varphi}\text{,}
\end{align*}
which shows that $R_{\xi_{\omega},\Omega_{\varphi}}$ extends to a unitary also.
Let $(x_{\lambda})$ be a bounded net in $N$ converging to $0$ in the strong
operator topology on $\mathbf{B}(L^{2}(N,\rho))$. By Eq. \eqref{Eq: ModuleIsometry},
\begin{align*}
\left\vert \langle x_{\lambda}y_{1}\xi_{\omega}z_{1},x_{\lambda}y_{2}
\xi_{\omega}z_{2})\rangle_{\omega}\right\vert  &  \leq\left\Vert \pi_{\omega
}(1_{N}\otimes(z_{1}z_{2}^{\ast})^{op})\right\Vert \left\Vert x_{\lambda}
y_{1}\xi_{\omega}\right\Vert _{\omega}\left\Vert x_{\lambda}y
\xi_{\omega}\right\Vert _{\omega}\\
&  =\left\Vert \pi_{\omega}(1_{N}\otimes(z_{1}z_{2}^{\ast})^{op})\right\Vert
\left\Vert x_{\lambda}y_{1}\Omega_{\rho}\right\Vert _{\rho}\left\Vert
x_{\lambda}y_{2}\Omega_{\rho}\right\Vert _{\rho}
\end{align*}
for all $y_{1},y_{2}\in N$ and $z_{1},z_{2}\in M$. By standard approximation
arguments, it follows that $\pi_{\omega}(x_{\lambda}\otimes1_{M}^{op})$
converges to $0$ in the strong operator topology on $\mathbf{B}(\mathcal{H}%
_{\omega})$. By Lemma 10.1.10 of \cite{KRII} it follows that the left action
of $N$ on $\mathcal{H}_{\omega}$ is normal. Normality of the right $M$ action
follows analogously.

Finally, the condition $\omega\circ(\alpha_{g}\otimes\beta_{g}^{op})=\omega$
for all $g\in G $ implies
\begin{align*}
&\langle \Pi_{g}(x_{1}\xi_{\omega}y_{1}), \Pi_{g}(x_{2}\xi_{\omega}y_{2})\rangle_{\omega}\\   &=\langle\pi_{\omega}(\alpha
_{g}(x_{1})\otimes \beta_{g}^{op}(y_{1}^{op}))\xi_{\omega},\pi_{\omega}(\alpha
_{g}(x_{2})\otimes \beta_{g}^{op}(y_{2}^{op}))\xi_{\omega}\rangle_{\omega}\\
  &=\omega(\alpha_{g}(x_{2}^{\ast}x_{1})\otimes\beta_{g}^{op}((y_{1}y_{2}^{\ast})^{op}))
=\omega(x_{2}^{\ast}x_{1}\otimes(y_{1}y_{2}^{\ast})^{op})\\
  &=\langle x_{1}\xi_{\omega}y_{1},x_{2}\xi_{\omega}y_{2}\rangle_{\omega}
\end{align*}
for all $x_{1},x_{2} \in N$ and $y_{1},y_{2} \in M$, so for every $g\in G $ the map $\Pi_{g}$ extends to a unitary on
$\mathcal{H}_{\omega}$. Similarly, $T_{t}$ for $t\in\mathbb{R}$ extends to a
unitary. It is easy to show that the maps $g \mapsto \Pi_{g}$ and $t\mapsto T_{t}$ are unitary representations, and strong continuity of these representations
follows by routine application of the triangle inequality. We omit the details.
\end{proof}

By Theorem \ref{BimoduleFromJoining} the correspondence $(\mathcal{H}_{\omega
},\xi_{\omega})$ carries extra structure. Theorem \ref{JoiningFromBimodule}
below establishes that any pointed correspondence with this extra structure
naturally gives rise to a joining. Together, these two theorems reveal that
the study of joinings is in fact a study of the natural class of pointed
correspondences we now specify.

\begin{definition}
\label{DefinitionofJoiningsasCorrespondences}Let $J_{h}(\mathfrak{N}
,\mathfrak{M})$ denote the set\footnote{ Here `h' stands for `Hilbert space'.} of pointed correspondences $(\mathcal{H},\xi)$
satisfying the following conditions:

\begin{enumerate}
\item \label{ExchangeMapIsomorphismCondition} The exchange maps $L_{\xi
,\Omega_{\rho}}$ and $R_{\xi,\Omega_{\varphi}}$ are well-defined and extend to Hilbert module isomorphisms $L_{\xi,\Omega_{\rho}}:\overline{N\xi}\rightarrow L^{2}(N,\rho)$ and \\ $R_{\xi,\Omega_{\varphi}}:\overline{\xi M}\rightarrow L^{2}(M,\varphi)$.

\item \label{UnitaryRepGammaandRealonH}For $g\in G $ and $t\in\mathbb{R}$
the maps $x\xi y\overset{\Pi_{g}}{\mapsto}\alpha_{g}(x)\xi\beta_{g}(y)$ and
$x\xi y\overset{T_{t}}{\mapsto}\sigma_{t}^{\rho}(x)\xi\sigma_{t}^{\varphi}(y)$
extend to unitary maps on $\mathcal{H}$, and $g\mapsto\Pi_{g}$ and $t\mapsto
T_{t}$ are strongly continuous unitary representations of $ G $ and
$\mathbb{R}$ on $\mathcal{H}$, respectively.
\end{enumerate}
\end{definition}

The fact that the maps in condition (1) above are unitary is somewhat delicate. For example, if the standard cyclic vector $\Omega_{\rho}$ of $L^{2}(N,\rho)$ is not a trace vector, then we cannot immediately deduce that the bimodule $L^{2}(N,\rho)$ is binormal if and only if $\Omega_{\rho}$ is a bounded vector as in \cite{Po}. It is therefore not immediate that one can identify cyclic correspondences with normal completely positive maps in the setting of infinite von Neumann algebras. Condition (1) above bypasses this problem, as we will see in the next section.

\begin{theorem}
\label{JoiningFromBimodule}Let $(\mathcal{H},\xi)\in J_{h}(\mathfrak{N}
,\mathfrak{M})$. Then the linear functional $\omega_{(\mathcal{H},\xi)}$ on
$N\odot M^{op}$ linearly extending $\omega_{(\mathcal{H},\xi)}(x\otimes
y^{op})=\langle x\xi y,\xi\rangle_{\mathcal{H}}$ is a joining of
$\mathfrak{N}$ and $\mathfrak{M}$, i.e. $\omega_{(\mathcal{H},\xi)}\in
J_{s}(\mathfrak{N},\mathfrak{M})$.
\end{theorem}

\begin{proof}
To reduce notation, let $\omega=\omega_{(\mathcal{H},\xi)}$. This linear functional is positive, being a vector state composed with a $*-$representation. By hypothesis,
\begin{align*}
\omega(x^{\ast}x\otimes1_{M}^{op})  &  =\norm{x\xi}_{\mathcal{H}}^{2}
=\norm{L_{\xi,\Omega_{\rho}}(x\xi)}_{\rho}^{2}\\
&  =\norm{xL_{\xi,\Omega_{\rho}}(\xi)}_{\rho}^{2}=\norm{x\Omega_{\rho}}_{\rho}
^{2}=\rho(x^{\ast}x)
\end{align*}
for all $x\in N$. It follows that $\omega(x\otimes1_{M}^{op})=\rho(x)$ for
every $x\in N$. Similarly, $\omega(1_{N}\otimes y^{op})=\varphi(y)$ for all
$y\in M$. Finally, by (\ref{UnitaryRepGammaandRealonH}) of Defn.
\ref{DefinitionofJoiningsasCorrespondences} it follows that for every $x\in N$
and $y\in M$
\begin{align*}
\omega(\alpha_{g}(x)\otimes\beta_{g}^{op}(y^{op}))  &  =\langle\alpha
_{g}(x)\xi\beta_{g}(1_{M}),\alpha_{g}(1_{N})\xi\beta_{g}(y)\rangle
_{\mathcal{H}}\\
&  =\langle x\xi,\xi y\rangle_{\mathcal{H}}=\omega(x\otimes y^{op})
\end{align*}
and again $\omega\circ(\alpha_{g}\otimes\beta_{g}^{op})=\omega$ for all
$g\in G $. The argument for $t\mapsto T_{t}$ is similar. Hence $\omega\in
J_{s}(\mathfrak{N},\mathfrak{M})$.
\end{proof}

\begin{remark}\label{GNSisomorphism}
Note that the map $\omega\mapsto(\mathcal{H}_{\omega},\xi_{\omega})$ from
$J_{s}(\mathfrak{N,M})$ into $J_{h}(\mathfrak{N,M})$ in Theorem
\ref{BimoduleFromJoining} is essentially the inverse of the map $(\mathcal{H}
,\xi)\mapsto\omega_{(\mathcal{H},\xi)}$ from $J_{h}(\mathfrak{N,M})$ into
$J_{s}(\mathfrak{N,M})$ in Theorem \ref{JoiningFromBimodule}. This is an
immediate consequence of the GNS construction.
\end{remark}

\section{Equivariant Markov Maps}

As before, let $(N,\rho)$ and $(M,\varphi)$ be von Neumann algebras with
faithful normal states. A unital completely positive (u.c.p.) map
$\Phi:N\rightarrow M$ that satisfies $\varphi\circ\Phi=\rho$ and $\sigma
_{t}^{\varphi}\circ\Phi=\Phi\circ\sigma_{t}^{\rho}$ for $t\in\mathbb{R}$ is
called a $(\rho,\varphi)$--\textbf{Markov map}. Given a $(\rho,\varphi)$--Markov
map $\Phi$, we perform the Stinespring construction to obtain a pointed
$N$--$M$ correspondence $(\mathcal{H}_{\Phi},\xi_{\Phi})$ via separation and
completion of $N\odot M$ with respect to the sesquilinear form 
\begin{align}\label{StinespringInnerProduct}
\langle
x_{1}\otimes y_{1},x_{2}\otimes y_{2}\rangle_{\Phi}:=\langle \Phi(x_{2}^{\ast}x_{1})\Omega_{\varphi}y_{1},\Omega_{\varphi}y_{2}\rangle_{\varphi},
\end{align}
and bimodule structure defined by
$x(x_{0}\otimes y_{0})y:=xx_{0}\otimes y_{0}y$ for $x,x_{0},x_{1},x_{2}\in N$
and $y,y_{0},y_{1},y_{2}\in M$. The distinguished unit cyclic vector
$\xi_{\Phi}$ in $\mathcal{H}_{\Phi}$ is the class of $1_{N}\otimes1_{M}$.
Let us remark that when $y_{1}, y_{2}$ are in the centralizer $M^{\varphi}$ then Eq.  \eqref{StinespringInnerProduct} may be written
\begin{align*}
\langle
x_{1}\otimes y_{1},x_{2}\otimes y_{2}\rangle_{\Phi}=\varphi(y_{2}^{*}\Phi(x_{2}^{*}x_{1})y_{1}),
\end{align*}
but this expression is not valid in general. A $(\rho,\varphi)$--Markov map $\Phi$ that satisfies $\beta_{g}\circ\Phi=\Phi
\circ\alpha_{g}$ for all $g\in G$ will be called a $G$--\textbf{equivariant Markov map}, or simply an equivariant Markov map, in what follows.
In the next theorem we show that every joining defined in \S3 has a natural associated equivariant Markov map.

\begin{theorem}
\label{Popa's Lemma copy(1)}For any $\omega\in J_{s}(\mathfrak{N}
,\mathfrak{M})$ there exists a normal $(\rho,\varphi)$--Markov map $\Phi_{\omega
}:N\rightarrow M$ such that $\beta_{g}\circ\Phi_{\omega}=\Phi_{\omega}
\circ\alpha_{g}$ for $g\in G $. The pointed $N$--$M$ bimodules
$(\mathcal{H}_{\omega},\xi_{\omega})$ and $(\mathcal{H}_{\Phi_{\omega}}
,\xi_{\Phi_{\omega}})$ are isomorphic.
\end{theorem}

\begin{proof}
Let $\pi_{N}$ and $\pi_{M^{op}}$ denote the representations of $N$ and $M^{op}$ on $\mathcal{H}_{\omega}$ and let $R=R_{\Omega_{\varphi},\xi_{\omega
}}:L^{2}(M)\rightarrow \overline{\xi_{\omega}M}$ denote the exchange map. By
Theorem \ref{BimoduleFromJoining} $R$ is unitary. For $x\in N$ consider
$\Phi_{\omega}(x)=R^{\ast}\pi_{N}(x)R$. A routine calculation shows that
\begin{align}\label{Eq:OmegaintoM}
\langle\Phi_{\omega}(x)(\Omega_{\varphi}zy),\Omega_{\varphi}w\rangle
_{\varphi} & =\langle(\Phi_{\omega}(x)(\Omega_{\varphi}z))y,\Omega_{\varphi}w\rangle_{\varphi}
\end{align} for every $y,z,w\in M$, hence
$R^{\ast}\pi_{N}(x)R\in M$ and $\Phi_{\omega}$ is a normal completely
positive map from $N$ into $M$. Note that $\Phi_{\omega}(1_{N})$ $=R^{\ast}R=$ $1_{M}$ and
$\varphi\circ\Phi_{\omega}(x)=\omega(x\otimes1_{M})=\rho(x)$ for $x\in N$. By (\ref{Eq: Diagonal Action Preservation}) of Defn. \ref{DefinitionJoiningasState} , for all
$x\in N$ and $y,z\in M$, 
\begin{align}\label{Eq:BreakthroughStuff}
\langle\Phi_{\omega}(\alpha_{g}(x))\Omega_{\varphi}y,\Omega_{\varphi}
z\rangle_{\varphi}  &  =\langle R^{\ast}\pi_{N}(\alpha_{g}(x))R\Omega
_{\varphi}y,\Omega_{\varphi}z\rangle_{\varphi}\\\nonumber
&  =\omega(\alpha_{g}(x)\otimes(yz^{\ast})^{op})\nonumber\\
&  =\omega(\alpha_{g}(x)\otimes\beta_{g}^{op}\beta_{g^{-1}}^{op}((yz^{\ast
})^{op})\nonumber\\
&  =\omega(x\otimes\beta_{g^{-1}}^{op}(yz^{\ast})^{op})\nonumber\\
&  =\langle \pi_{N}(x) \pi_{M^{op}}(\beta_{g^{-1}}^{op}(yz^{\ast})^{op}
)\xi_{\omega},\xi_{\omega}\rangle_{\omega}\nonumber\\
&  =\langle \pi_{N}(x)RV_{g}^{\ast}J_{\varphi}(zy^{\ast})J_{\varphi}
V_{g}\Omega_{\varphi},R\Omega_{\varphi} \rangle_{\omega}\nonumber\\
&  =\langle R^{\ast}\pi_{N}(x)RV_{g}^{\ast}J_{\varphi}(zy^{\ast})J_{\varphi}\Omega_{\varphi},\Omega_{\varphi} \rangle_{\omega}\nonumber\\
&  =\langle V_{g}\Phi_{\omega}(x)V_{g}^{\ast}J_{\varphi}(zy^{\ast})J_{\varphi}\Omega_{\varphi},V_{g}\Omega_{\varphi} \rangle_{\omega}\nonumber\\
&  =\langle V_{g}\Phi_{\omega}(x)V_{g}^{\ast}J_{\varphi}(zy^{\ast})J_{\varphi}\Omega_{\varphi},\Omega_{\varphi} \rangle_{\omega}\nonumber\\
&  =\langle\beta_{g}(\Phi_{\omega}(x))\Omega_{\varphi}y,\Omega_{\varphi
}z\rangle_{\varphi},\nonumber
\end{align}
and therefore $\Phi_{\omega}(\alpha_{g}(x))=\beta_{g}(\Phi_{\omega}(x))$ for
all $g\in G $. Analogously, $\Phi_{\omega}(\sigma_{t}^{\rho}(x))=\sigma
_{t}^{\varphi}(\Phi_{\omega}(x))$ for all $t\in\mathbb{R}$.

For all $x_{1},x_{2}\in N$ and $y_{1},y_{2}\in M$
\begin{align}
\langle x_{1}\xi_{\Phi_{\omega}}y_{1},x_{2}\xi_{\Phi_{\omega}}y_{2}
\rangle_{\Phi_{\omega}}  & =\langle\Phi_{\omega}(x_{2}^{\ast}x_{1})\Omega_{\varphi}y_{1}%
,\Omega_{\varphi}y_{2}\rangle_{\varphi}\nonumber\label{Eq: BimExchWD}\\
&  =\langle(R^{\ast}\pi_{N}(x_{2}^{\ast}x_{1})R\Omega_{\varphi})y_{1}
,\Omega_{\varphi}y_{2}\rangle_{\varphi}\nonumber\\
&  =\langle R^{\ast}\pi_{N}(x_{2}^{\ast}x_{1})R\Omega_{\varphi},\Omega
_{\varphi}y_{2}y_{1}^{\ast}\rangle_{\varphi}\nonumber\\
&  =\langle x_{1}\xi_{\omega}y_{1},x_{2}\xi_{\omega}y_{2}\rangle_{\omega
}\text{.}\nonumber
\end{align}
Therefore, $x\xi_{\omega}y\mapsto x\xi_{\Phi_{\omega}}y$ for $x\in
N$, $y\in M$ extends to an isomorphism of correspondences
\begin{equation*}
W:\mathcal{H}_{\omega}\rightarrow\mathcal{H}_{\Phi_{\omega}}\text{.}
\label{Eq: BimoduleExchangeMap}
\end{equation*}
\end{proof}

 In the remainder of the paper, for a given $\omega \in J_{s}(\mathfrak{N},\mathfrak{M})$, the notation $\Phi_{\omega}$ will denote the completely positive map constructed above in the proof of Theorem \ref{Popa's Lemma copy(1)}.

\begin{definition}
\label{JoiningMarkovDefinition}Let $J_{m}(\mathfrak{N},\mathfrak{M})=\{\Phi:N\rightarrow M |\Phi$
is $(\rho,\varphi)$--Markov, $\beta_{g}\circ\Phi=\Phi\circ\alpha_{g}$ for $g\in G \}$\footnote{ Here, `m' stands for `Markov map'.}. For $\Phi\in J_{m}(\mathfrak{N},\mathfrak{M})$, let $\omega_{\Phi
}$ denote the state on $N\odot M^{op}$ extending $\omega_{\Phi}(x\otimes
y^{op})=\langle x\xi_{\Phi}y,\xi_{\Phi}\rangle_{\Phi}$, for $x\in N$, $y\in M$\footnote{Again, the functional $\omega_{\Phi}$ is positive since it is the composition of a state and a $*-$representation, this time the representation coming from the Stinespring construction}.
\end{definition}

The following shows that every equivariant Markov map between two systems
naturally gives a joining.

\begin{theorem}
\label{CharacterizeJoins} If $\Phi\in J_{m}(\mathfrak{N},\mathfrak{M})$ then
$\omega_{\Phi}\in J_{s}(\mathfrak{N},\mathfrak{M})$.
\end{theorem}

\begin{proof}
For $x\in N$ and $y\in M$,
\begin{equation*}
\omega_{\Phi}(x\otimes1_{M}^{op})=\langle x\xi_{\Phi},\xi_{\Phi}\rangle_{\Phi
}=\varphi(\Phi(x))=\rho(x) \label{Eq: CharacterizeJoinsFirstEq}
\end{equation*}
and
\begin{equation}\label{Eq: CharacterizeJoinsSecondEqn}
\omega_{\Phi}(1_{N}\otimes y^{op})=\langle\xi_{\Phi}y,\xi_{\Phi}\rangle_{\Phi
}=\varphi(y). 
\end{equation}
For all $g\in G $,
\begin{align}\label{Eq: CharacterizeJoinsThirdEqn}
\omega_{\Phi}(\alpha_{g}(x)\otimes(\beta^{op}_{g})(y^{op}))  &  =\langle\alpha
_{g}(x)\xi_{\Phi}\beta^{op}_{g}(y),\xi_{\Phi}\rangle_{\Phi}
\\
&  =\langle \Phi(\alpha_{g}(x))\Omega_{\varphi}\beta_{g}(y),\Omega_{\varphi} \rangle_{\varphi} \nonumber\\
&  =\langle \Phi(\alpha_{g}(x))(\beta_{g}(y))^{op}\Omega_{\varphi},\Omega_{\varphi} \rangle_{\varphi} \nonumber\\
&  =\langle \beta_{g}(\Phi(x))(\beta_{g}(y))^{op}\Omega_{\varphi},\Omega_{\varphi} \rangle_{\varphi} \nonumber\\
&  =\langle \beta_{g}(\Phi(x))\beta_{g}^{op}(y^{op})\Omega_{\varphi},\Omega_{\varphi} \rangle_{\varphi} \nonumber\\
&  =\langle V_{g}(\Phi(x))V_{g}^{\ast}V_{g}J_{\varphi}y^{\ast}J_{\varphi}V_{g}^{\ast}\Omega_{\varphi},\Omega_{\varphi} \rangle_{\varphi} \nonumber\\
&  =\langle \Phi(x)J_{\varphi}y^{\ast}J_{\varphi}\Omega_{\varphi},\Omega_{\varphi} \rangle_{\varphi} \nonumber\\
&  =\omega_{\Phi}(x\otimes y^{op}),\nonumber
\end{align}
and therefore $\omega_{\Phi}\circ(\alpha_{g}\otimes\beta_{g}^{op})=\omega$.
Similarly, $\omega\circ(\sigma_{t}^{\rho}\otimes(\sigma_{t}^{\varphi}
)^{op})=\omega$\\ for $t\in\mathbb{R}$.
\end{proof}

We now prove that $\omega\mapsto\Phi_{\omega}$ and $\Phi\mapsto\omega_{\Phi}$
are inverses of one another, providing the alternative point of view of
joinings as equivariant Markov maps. The benefit of this translation will
become especially evident in \S\ref{Erg}, where working with Markov maps (rather than states on an algebraic tensor product) is seen to be the more natural approach to proving noncommutative
analogues of classical joining results.  

\begin{theorem}
\label{CharacterizeReverse} For all $\omega\in J_{s}(\mathfrak{N}
,\mathfrak{M})$ and $\Phi\in J_{m}(\mathfrak{N},\mathfrak{M})$, $\omega
_{\Phi_{\omega}}=\omega$ and $\Phi_{\omega_{\Phi}}=\Phi$.
\end{theorem}

\begin{proof}
Let $\omega\in J_{s}(\mathfrak{N},\mathfrak{M})$. Then for $x\in N$ and $y\in
M$,
\begin{align}
\omega_{\Phi_{\omega}}(x\otimes y^{op})  &  =\langle x\xi_{\Phi_{\omega}}
y,\xi_{\Phi_{\omega}}\rangle_{\Phi_{\omega}}\\
&  =\langle \Phi_{\omega}(x)\Omega_{\varphi}y,\Omega_{\varphi} \rangle \nonumber\\
& =\langle (R^{*}\pi_{N}(x)R)\Omega_{\varphi}y,\Omega_{\varphi} \rangle \nonumber\\
& =\langle (\pi_{N}(x)R)\Omega_{\varphi}y, R \Omega_{\varphi} \rangle \nonumber\\
& =\langle x \xi_{\omega}y, \xi_{\omega} \rangle \nonumber \\
&  =\omega(x\otimes y^{op}),\nonumber
\end{align}
and hence $\omega_{\Phi_{\omega}}=\omega$.

Let $\Phi\in J_{m}(\mathfrak{N},\mathfrak{M})$ and $R=R_{\Omega_{\varphi}
,\xi_{\omega_{\Phi}}}$. For $x\in N$ and $y,z\in M$,
\begin{align}\label{Eq: CharacterizeRevSecondEqn}
\langle\Phi_{\omega_{\Phi}}(x)\Omega_{\varphi}y,\Omega_{\varphi}
z\rangle_{\varphi}  &  =\langle R^{\ast}\pi_{N}(x)R\Omega_{\varphi}
y,\Omega_{\varphi}z\rangle_{\varphi}\\
&  =\langle x\xi_{\omega_{\Phi}}y,\xi_{\omega_{\Phi}}z\rangle_{\omega_{\Phi}
}\nonumber\\
&  =\omega_{\Phi}(x\otimes(yz^{\ast})^{op})\nonumber\\
&  =\langle x\xi_{\Phi}y,\xi_{\Phi}z\rangle_{\Phi}\nonumber\\
&  =\langle\Phi(x)\Omega_{\varphi}y,\Omega_{\varphi}z\rangle_{\varphi
},\nonumber
\end{align}
therefore, $\Phi_{\omega_{\Phi}}=\Phi$.
\end{proof}

The above results together yield the following fact.

\begin{corollary}
\label{MarkovisNormal}Every $\Phi\in J_{m}(\mathfrak{N},\mathfrak{M})$ is a normal map.
\end{corollary}

\begin{proof}
By Theorem \ref{CharacterizeJoins}, $\omega_{\Phi}\in J_{s}(\mathfrak{N}
,\mathfrak{M})$, by Theorem \ref{CharacterizeReverse}, $\Phi_{\omega_{\Phi}
}=\Phi$ and by Theorem \ref{Popa's Lemma copy(1)} $\Phi_{\omega_{\Phi}}$ is normal.
\end{proof}

The next theorem summarizes the results of \S3 and the results of the present section thus far and adds (3), which will not be used later in this paper.  Together these show precisely how a joining of
two systems $\mathfrak{N}$ and $\mathfrak{M}$ can be viewed alternatively as an element of
one of the three sets $J_{s}(\mathfrak{N,M})$, $J_{h}(\mathfrak{N,M})$ or
$J_{m}(\mathfrak{N,M})$.

\begin{theorem}
\label{JoiningDictionary} Let $\mathfrak{N}=(N,\rho,\alpha, G )$ and
$\mathfrak{M}=(M,\varphi,\beta, G )$ be $W^{\ast}$--dynamical systems.

\begin{enumerate}
\item \label{JoinDictStateBimoduleSummary}For all $\omega\in J_{s}
(\mathfrak{N,M})$, the GNS Hilbert space $(\mathcal{H}_{\omega},\xi_{\omega
})\in J_{h}(\mathfrak{N,M})$, for all $(\mathcal{H},\xi)\in J_{h}
(\mathfrak{N,M})$ the vector state $\omega_{\xi}\in J_{s}(\mathfrak{N,M})$ and
furthermore, $\omega_{\xi_{\omega}}=\omega$ and $(\mathcal{H}_{\omega_{\xi}
},\xi_{\omega_{\xi}})\cong(\mathcal{H},\xi)$ as pointed $N$--$M$ correspondences.

\item \label{JoinDictStateMarkovSummary}For all $\omega\in J_{s}
(\mathfrak{N,M})$, the $(\rho,\varphi)$-Markov map $\Phi_{\omega}\in
J_{m}(\mathfrak{N,M})$, for all $\Phi\in J_{m}(\mathfrak{N,M})$ the vector state
$\omega_{\Phi}\in J_{s}(\mathfrak{N,M})$, and furthermore $\Phi_{\omega_{\Phi
}}=\Phi$ and $\omega_{\Phi_{\omega}}=\omega$.

\item \label{JoinDictMarkovBimoduleSummary}For all $\Phi\in J_{m}(\mathfrak{N,M}%
)$, the Stinespring bimodule $(\mathcal{H}_{\Phi},\xi_{\Phi})\in
J_{h}(\mathfrak{N,M})$, for all $(\mathcal{H},\xi)\in J_{h}(\mathfrak{N,M})$
the associated completely positive map $\Phi_{\xi}(\cdot)=R_{\Omega_{M},\xi
}^{\ast}\pi_{N}(\cdot)R_{\Omega_{M},\xi}\in J_{m}(\mathfrak{N,M})$ and furthermore
$\Phi_{\xi_{\Phi}}=\Phi$ and $(\mathcal{H}_{\Phi_{\xi}},\xi_{\Phi_{\xi}}
)\cong(\mathcal{H},\xi)$ as pointed correspondences.
\end{enumerate}
\end{theorem}

\begin{proof}
The first statement follows directly from Theorem \ref{BimoduleFromJoining} and Theorem \ref{JoiningFromBimodule} once one checks that $\omega_{\xi_{\omega}}=\omega$ and that the bimodule map sending $\xi_{\omega_{\xi}}$ to $\xi$ extends to a unitary. The second statement summarizes Theorems \ref{Popa's Lemma copy(1)}, \ref{CharacterizeJoins} and \ref{CharacterizeReverse}.

It remains to prove the third statement. Let $\Phi\in J_{m}(\mathfrak{N,M})$ and let $(\mathcal{H}_{\Phi},\xi_{\Phi})$ be the associated Stinespring bimodule. We must show that this bimodule is in $J_{h}(\mathfrak{N,M})$. The left exchange map $L_{\xi_{\Phi},\Omega_{\rho}}$, if well-defined, will be a left $N$--module map $L_{\xi_{\Phi},\Omega_{\rho}}:N\xi_{\Phi}\rightarrow L^{2}(N,\rho)$. We must show this map is well--defined and extends to a unitary. Suppose $x \in N$ and $x \xi_{\Phi}=0$, then
\begin{align*}
\langle x \xi_{\Phi}, x \xi_{\Phi} \rangle=\phi(\Phi(x^{*}x))
& =\rho(x^{*}x)=0
\end{align*}

and hence $x=0$, and $L_{\xi_{\Phi},\Omega_{\rho}}$ is well--defined. An analogous computation establishes that the map extends to a unitary. We omit the similar argument needed for right exchange maps. Let $g\in G$ be given. Then 
\begin{align*}
\langle \Pi_{g}(x \xi_{\Phi} y), \Pi_{g} (z \xi_{\Phi} w) \rangle_{\Phi}
& = \langle \alpha_{g}(x) \xi_{\Phi} \beta_{g}(y), \alpha_{g}(z) \xi_{\Phi} \beta_{g}(w) \rangle_{\Phi} \\
& = \langle \Phi(\alpha_{g}(z^{*}x)) \Omega_{\varphi} \beta_{g}(y), \Omega_{\varphi} \beta_{g}(w) \rangle_{\varphi} \nonumber \\
& = \langle \beta_{g}(\Phi(z^{*}x)) \Omega_{\varphi} \beta_{g}(y), \Omega_{\varphi} \beta_{g}(w) \rangle_{\varphi} \nonumber \\
& = \langle V_{g}((\Phi(z^{*}x)) \Omega_{\varphi} y, V_{g}(\Omega_{\varphi} w) \rangle_{\varphi} \nonumber \\
& = \langle \Phi(z^{*}x) \Omega_{\varphi} y, \Omega_{\varphi} w \rangle_{\varphi} \nonumber \\
& = \langle x \xi_{\Phi} y, z \xi_{\Phi} w \rangle_{\Phi} \nonumber
\end{align*}
establishing that $\Pi_{g}$ extends to a unitary. Proofs of the similar result for $T_{t}$ and the statement that the associated unitary representations are strongly continuous are omitted.

Assume now that $(\mathcal{H},\xi)\in J_{h}(\mathfrak{N,M})$. Then $\Phi_{\xi}(\cdot)=R_{\Omega_{M},\xi
}^{\ast}\pi_{N}(\cdot)R_{\Omega_{M},\xi}$ is obviously a normal completely positive map. Furthermore,
\begin{align*}
R_{\Omega_{M},\xi}^{\ast}\pi_{N}(1_N)R_{\Omega_{M},\xi}(\Omega_{M}z)
= R_{\Omega_{M},\xi}^{\ast}(\xi z)=\Omega_{M} z
\end{align*}
for all $z \in M$, and
\begin{align*}
\varphi(\Phi_{\xi}(x))=\langle R_{\Omega_{M},\xi}^{\ast}\pi_{N}(x)R_{\Omega_{M},\xi} \Omega_{M}, \Omega_{M} \rangle = \langle x \xi, \xi \rangle =\rho (x) \nonumber
\end{align*}
for all $x \in N$, so the map is unital and preserves states. Given $g \in G$ we have for all $x \in N$ and $y,z \in M$ that
\begin{align*}
&\langle \Phi_{\xi}(\alpha_{g}(x))\Omega_{M} y, \Omega_{M} z \rangle\\ & =
\langle R_{\Omega_{M},\xi}^{\ast}\pi_{N}(\alpha_{g}(x))R_{\Omega_{M},\xi}\Omega_{M} y, \Omega_{M} z \rangle \\
& =\langle \alpha_{g}(x)\xi y, \xi z \rangle \nonumber \\
& =\langle x\xi \beta_{g^{-1}}(y), \xi \beta_{g^{-1}}(z) \rangle \nonumber\\
& = \langle R_{\Omega_{M},\xi}^{\ast}\pi_{N}(x)R_{\Omega_{M},\xi}\Omega_{M} \beta_{g^{-1}}(y), \Omega_{M} \beta_{g^{-1}}(z) \rangle \nonumber \\
& = \langle V_{g}(R_{\Omega_{M},\xi}^{\ast}\pi_{N}(x)R_{\Omega_{M},\xi}\Omega_{M} \beta_{g^{-1}}(y)), V_g(\Omega_{M} \beta_{g^{-1}}(z)) \rangle \nonumber \\
& = \langle  \beta_{g}(\Phi_{\xi}(x))\Omega_{M} y, \Omega_{M} z \rangle
\end{align*}
establishing that $\Phi_{\xi}$ intertwines the $G$--actions. The argument is identical for the Markov property. Finally, the bimodule map sending the cyclic vector $\xi_{\Phi_{\xi}}$ to the cyclic vector $\xi$ is readily checked to extend to a unitary and the proof of (3) is complete.
\end{proof}
Physically, quantum decoherence is the loss of quantum information in a
channel due to its interaction with the environment. In the absence of quantum decoherence, an isolated nonrelativistic quantum system is modelled by an operator algebra of observables
which evolves in time by a one--parameter family of state--preserving
automorphisms implemented by unitaries. A coherent quantum channel between two
systems seamlessly patches together their time evolutions. Coherence of
quantum channels is mathematically modelled by Markov maps as above. The fact
that there should be no quantum information loss in a coherent quantum channel
is further reflected in the following result of Accardi and Cecchini (cf.
Proposition 6.1 of \cite{AC} or Lemma 2.5 of \cite{AD}): Let $\Phi
:N\rightarrow M$ be a normal unital completely positive map. Then there
exists a normal unital completely positive map $\Phi^{\ast}:M\rightarrow N$
satisfying
\begin{equation}
\rho(\Phi^{\ast}(y)x)=\varphi(y\Phi(x))\label{Eq: AccardiCecciniAdjoint}
\end{equation}
for all $y\in M$ and $x\in N$ if and only if $\varphi\circ\Phi=\rho$ and
$\Phi\circ\sigma_{t}^{\rho}=\sigma_{t}^{\varphi}\circ\Phi$ for all
$t\in\mathbb{R}$. Our next theorem ensures that if $\Phi$ is additionally an
equivariant quantum channel, then the adjoint channel $\Phi^{\ast}$ will also be
equivariant.

\begin{theorem}
\label{AdjointisEquivariant}If $\Phi\in J_{m}(\mathfrak{N},\mathfrak{M})$ then $\Phi^{*}\in J_{m}(\mathfrak{N},\mathfrak{M}).$
\end{theorem}

\begin{proof}
Assume $\Phi\in J_{m}(\mathfrak{N},\mathfrak{M})$. Then, in particular, $\Phi\circ\alpha_{g}=\beta_{g}\circ\Phi$ for all $g\in G $ and for
$g\in G $, $x\in N$ and $y\in M$
\begin{align*}
\rho(\Phi^{\ast}(\beta_{g}(y))x)  &  =\varphi(\beta_{g}
(y)\Phi(x))\label{Eq: AdjointisEquivariantCalculation}\\
&  =\varphi(y\beta_{g^{-1}}(\Phi(x)))\nonumber\\
&  =\varphi(y\Phi(\alpha_{g^{-1}}(x))\nonumber\\
&  =\rho(\Phi^{\ast}(y)\alpha_{g^{-1}}(x))\nonumber\\
&  =\rho(\alpha_{g}(\Phi^{\ast}(y))x)\text{.}\nonumber
\end{align*}
But $\rho$ is faithful, so the result follows.
\end{proof}

\section{Some Basic Examples}\label{JOp}

We now provide the noncommutative analogues of several standard examples of joinings. Since most of these noncommutative examples are abstruse as states on a tensor product but can be built using familiar Markov maps, we present the examples as elements of $J_{m}(\mathfrak{N,M})$. One easily translates
to $J_{s}(\mathfrak{N,M})$ and $J_{h}(\mathfrak{N,M})$ using Theorem
\ref{JoiningDictionary}.

\subsection{The Trivial Joining}

As in the classical case, the systems $\mathfrak{N}$ and $\mathfrak{M}$ always
admit the trivial joining. In the language of correspondences this is the
\textit{coarse correspondence}.

\begin{definition}\label{DisjointDefinition}
The map $\rho(\cdot)1_{M}\in J_{m}(\mathfrak{N,M})$ is the\textbf{ trivial
joining} of $\mathfrak{N}$ and $\mathfrak{M}$. The systems $\mathfrak{N}$ and
$\mathfrak{M}$ are \textbf{disjoint} if $J_{m}(\mathfrak{N,M})=\{\rho(\cdot
)1_{M}\}$. 
\end{definition}

Note that $\varphi(\cdot)1_{N}=(\rho(\cdot)1_{M})^{\ast}$, so $\mathfrak{N}$
and $\mathfrak{M}$ are disjoint if and only if $\mathfrak{M}$ and
$\mathfrak{N}$ are disjoint.

\subsection{Graph Joinings and Subsystems}

A graph joining of two classical systems is a measure that is concentrated on
the graph of a homomorphism of one system into the other.

In the noncommutative case we obtain the following definition.

\begin{definition}
\label{GraphJoiningIsomorph}A \textbf{graph joining} of $\mathfrak{N}$ and
$\mathfrak{M}$ is a $\ast$ --homomorphism $\iota\in J_{m}(\mathfrak{N,M})$.
If $\iota$ is bijective, then it is an \textbf{isomorphism graph joining.}
\end{definition}

The \textit{identity correspondence} is an example of an isomorphism graph self-joining of any system $\mathfrak{N}$. We note that the quantum diagonal measures considered in \cite{Fi} are isomorphism graph joinings sans modular invariance.

Our graph joinings may be considered the modular invariant versions of the *-- homomorphisms that define factors in Definition 3.2 of \cite{Du}. We mention that Construction 3.4 of \cite{Du} begins with an equivariant injective $*-$homomorphism from one system to another and produces a joining (in the sense of \cite{Du}) of the second system with the opposite of the first. The corresponding completely positive map is the generalized conditional expectation that is bidual (cf. \cite{AC}) to the inclusion defined by the injective $*-$homomorphism.  Whether we use the original $*-$homomorphism or the result of Construction 3.4 of \cite{Du}, an opposite algebra is involved. This observation, together with the fact that there are von Neumann algebras which are not antiisomorphic to themselves (cf. \cite{Co2}), including the opposite algebra in the data used to define a joining appears necessary. Otherwise, an appropriate notion of graph self joining for systems arising from actions on such von Neumann algebras is not readily available. 

The following definition of subsystem is similar to the definition of modular subsystem in \cite{Du3}, the difference again being the presence of the opposite algebra. As in \cite{Du3}, the modular condition is needed to ensure the existence of a unique state--preserving conditional expectation.

\begin{definition}
\label{Subsystem}A \textbf{subsystem }of $\mathfrak{M}$ is a pair
$(\mathfrak{N},\iota)$, where $\iota\in J_{m}(\mathfrak{N,M})$ is an injective graph joining.
If $N\subset M$ we may say $\mathfrak{N}$ is a subsystem of $\mathfrak{M}$,
the inclusion being understood.
\end{definition}

It is easy to check that if $\iota\in J_{m}(\mathfrak{N,M})$ is an injective graph joining then
$\iota^{-1}:\iota(N)\rightarrow N$ is a unital $\ast$-isomorphism satisfying
$\rho\circ\iota^{-1}=\varphi$, $\iota^{-1}\circ\sigma_{t}^{\varphi}=\sigma
_{t}^{\rho}\circ\iota^{-1}$ and $\alpha_{g}\circ\iota^{-1}=\iota^{-1}%
\circ\beta_{g}$ for all $t\in\mathbb{R}$ and $g\in G $. 

The following is a slight modification of Example 2.7 (c) of \cite{AD}. This
is an important fact, so we include a proof for completeness.

\begin{theorem}
\label{ConditionalExpectationisMarkov}If $\iota\in J_{m}(\mathfrak{N,M})$ is an injective graph joining then there is a unique faithful normal conditional expectation
$\mathbb{E}_{\iota(N)}$ of $M$ onto $\iota(N)$ satisfying $\varphi \circ \mathbb{E}_{\iota(N)}=\varphi$, and $\iota^{*}=\iota^{-1}\circ
\mathbb{E}_{_{\iota(N)}}\in J_{m}(\mathfrak{M,N})$.
\end{theorem}

\begin{proof}
By hypothesis $\sigma_{t}^{\varphi}(\iota(N))\subset N$ for all $t\in
\mathbb{R}$. By \cite{Ta} there exists a unique faithful, normal conditional
expectation $\mathbb{E}_{\iota(N)}$ of $M$ onto $\iota(N)$ such that
$\varphi\circ\mathbb{E}_{\iota(N)}=\varphi$. Notice that for every $y\in M$
and $x\in N$%
\begin{align}
\rho((\iota^{-1}(\mathbb{E}_{\iota(B)}(y)))x) &  =\rho((\iota^{-1}%
(\mathbb{E}_{\iota(N)}(y)))\iota^{-1}\iota(x))\\\nonumber
&  =\rho(\iota^{-1}(\mathbb{E}_{\iota(N)}(y)\iota(x)))\\\nonumber
&  =\varphi(\mathbb{E}_{\iota(N)}(y)\iota(x))\\\nonumber
&  =\varphi(\mathbb{E}_{\iota(N)}(y\iota(x)))\\\nonumber
&  =\varphi(y\iota(x)),
\end{align}
yielding $\iota^{\ast}=\iota^{-1}\circ\mathbb{E}_{\iota(N)}$. By Proposition
\ref{AdjointisEquivariant} we have $\iota^{-1}\circ\mathbb{E}_{\iota(N)}\in
J_{m}(\mathfrak{M,N})$.
\end{proof}

A \textbf{group subsystem} of $\mathfrak{N}$ is a subsystem of the form $(N^{K},\rho|_{N^{K}},\alpha, G )$ where
$K$ is a closed subgroup of $Aut^{ G }(N)$, the group of $ G 
$-equivariant automorphisms of $N$ that preserve $\rho$, and $N^{K}$ is the fixed point algebra of $K$ in $N$. We also remark that
if $M$ is a factor of type $\rm{III}_{\lambda}$ with $\lambda \in (0,1)$ then $\mathfrak{M}$
automatically has a nontrivial subsystem coming from the centralizer
$M^{\varphi}$ of the state $\varphi$.

\subsection{Relatively Independent Joinings over a Common Subsystem}

Two systems $\mathfrak{N}$ and $\mathfrak{M}$ are said to have $\mathfrak{B}$
as a common subsystem if there exist injective graph joinings $\iota_{N}\in
J_{m}(\mathfrak{B},\mathfrak{N})$ and $\iota_{M}\in J_{m}(\mathfrak{B},\mathfrak{M})$
such that $(\mathfrak{B},\iota_{N})$ is a subsystem of $\mathfrak{N}$ and
$(\mathfrak{B},\iota_{M})$ is a subsystem of $\mathfrak{M}$. 

\begin{definition}
The \textbf{relatively independent joining} of $\mathfrak{N}$ and
$\mathfrak{M}$ over $\mathfrak{B}$ with respect to the injective graph joinings
$\iota_{N}\in J_{m}(\mathfrak{B},\mathfrak{N})$ and $\iota_{M}\in J_{m}(\mathfrak{B}%
,\mathfrak{M})$ is $\iota_{M}\circ\iota_{N}^{\ast}$.
\end{definition}

Notice that $\iota_{N}\circ\iota_{M}^{\ast}=(\iota_{M}\circ\iota_{N}^{\ast
})^{\ast}$. By Theorem 6.8 of \cite{Gl} the above generalizes the classical notion of relatively independent joining.

\section{Ergodicity and Weak Mixing}\label{Erg}

In this section, we provide characterizations of classical notions like
ergodicity, primeness and weak mixing in terms of joinings.  Theorem 
\ref{OtherwayWmDisjointCompact} and \ref{CharacterizeWeakMixing} are the main results of this 
section, which establish that a system disjoint from every compact system is weakly mixing. Central to these results is the fact that one can control the invariant subspaces of the unitary representation associated to an ergodic system $\mathfrak{M}$ -- specifically, we will show that they all come from the underlying von Neumann algebra itself and, in particular, from the centralizer of the invariant state for the system.  This is accomplished by our main technical tool in Theorem 6.10, which is a variation of the main result of \cite{HLS} that we find to be of independent interest. 
 
We first need some preparation.

\subsection{Intertwiners of Representations Induced by Equivariant 
Markov Maps}

As before, let $N$ and $M$ be von Neumann algebras equipped with faithful normal states
$\rho$ and $\varphi$ respectively. Let $\Phi:N\rightarrow M$ be a u.c.p. map such that $\varphi\circ\Phi=\rho$. Then $\Phi$ is normal 
$($see Corollary  \ref{MarkovisNormal}$)$. Define $T_{\Phi}:L^{2}(N,\rho)\rightarrow L^{2}(M,\varphi)$ by $T_{\Phi}(x\Omega_{\rho}%
)=\Phi(x)\Omega_{\varphi}$ for all $x\in N$. The operator $T_{\Phi}$ is, a priori, an unbounded
operator. However, by Kadison's inequality we have
\begin{align}\label{T_phi}
\langle T_{\Phi}(x\Omega_{\rho}),T_{\Phi}(x\Omega_{\rho})\rangle_{\varphi}  &
=\langle\Phi(x)\Omega_{\varphi},\Phi(x)\Omega_{\varphi}\rangle_{\varphi}\\
\nonumber&  =\varphi(\Phi(x^{\ast})\Phi(x))\\
\nonumber&  \leq\varphi(\left\Vert \Phi(1)\right\Vert \Phi(x^{\ast}x))\\
\nonumber&  =\rho(x^{\ast}x)\\
\nonumber&  =\langle x\Omega_{\rho},x\Omega_{\rho}\rangle_{\rho}.
\end{align}
Thus, $T_{\Phi}$ extends to a bounded operator from $L^{2}(N,\rho)$ to
$L^{2}(M,\varphi)$ of norm $1$, as $\norm{T_{\Phi}(\Omega_{\rho})}_{\varphi}=1$. Also 
note that from Eq. \eqref{Eq: AccardiCecciniAdjoint} it follows that 
\begin{align}\label{adjointT_phi}
T_{\Phi}^{*} = T_{\Phi^{*}}.
\end{align}

The following observation will be used frequently in this section and beyond. Let
$\mathfrak{N}$ and $\mathfrak{M}$ be systems, and let $\Phi\in J_{m}(\mathfrak{N}
,\mathfrak{M})$. Since $\Phi$ is $ G $--equivariant and a $(\rho,\varphi)$--Markov map, it follows that $\Phi
\circ\alpha_{g}=\beta_{g}\circ\Phi$ and $\Phi
\circ\sigma_{t}^{\rho}=\sigma_{t}^{\varphi}\circ\Phi$ for all $g\in G $ and $t\in\mathbb{R}$. Thus, for $x\in N$
and $g\in G $, we have
\begin{align}
&  \Phi(\alpha_{g}(x))\Omega_{\varphi}=\beta_{g}(\Phi(x))\Omega_{\varphi
},\text{ and equivalently, }\\\nonumber
&  T_{\Phi}U_{g}(x\Omega_{\rho})=V_{g}T_{\Phi}(x\Omega_{\rho}).\\\nonumber
\text{Similarly, for $t \in \mathbb{R}$ we have }&\\
&  \Phi(\sigma_{t}^{\rho}(x))\Omega_{\varphi}=\sigma_{t}^{\varphi}(\Phi(x))\Omega_{\varphi
},\text{ and equivalently, }\\\nonumber
&  T_{\Phi}\Delta^{it}_{\rho}(x\Omega_{\rho})=\Delta^{it}_{\varphi}T_{\Phi}(x\Omega_{\rho}).
\end{align}
Consequently, for all $g\in G$ and $t \in \mathbb{R},$ we obtain
\begin{align}\label{Eq: Crucial Equation}
&T_{\Phi}U_{g}=V_{g}T_{\Phi};\\
\nonumber&T_{\Phi}\Delta^{it}_{\rho}=\Delta^{it}_{\varphi}T_{\Phi},
\end{align}
as well as 
\begin{align}\label{Eq: Crucial Equation1}
&T_{\Phi}U_{g}\Delta^{it}_{\rho}=V_{g}\Delta^{it}_{\varphi}T_{\Phi};\\
\nonumber&T_{\Phi}\Delta^{it}_{\rho}U_{g}=\Delta^{it}_{\varphi}V_{g}T_{\Phi}.
\end{align}
Note that since $\sigma_{t}^{\rho}\alpha_{g}=\alpha_{g}\sigma_{t}^{\rho}$ and  
$\sigma_{t}^{\varphi}\beta_{g}=\beta_{g}\sigma_{t}^{\varphi}$ for all 
$g\in  G $ and $t\in \mathbb{R}$, and since the vacuum vectors are fixed by $U_{g}$ and $V_{g}$, both of the expressions in Eq. \eqref{Eq: Crucial Equation1} carry the same information. 
Let $\pi_{U}\cdot\pi_{{\Delta}_{\rho}}$ and $\pi_{V}\cdot\pi_{{\Delta}_{\varphi}}$ be the representations
of $ G \times \mathbb{R}$ given by $(g,t)\mapsto U_{g}\Delta_{\rho}^{it}$ and 
$(g,t)\mapsto V_{g}\Delta_{\varphi}^{it}$, respectively. The above discussion can be formally 
summarized in the following  proposition.

\begin{proposition}\label{Prop:Intertwine}
The operator $T_{\Phi}$ intertwines the representations$:$\\
$(i)$ $\pi_{U}$ and $\pi_{V}$ of $ G $;\\
$(ii)$ $\pi_{{\Delta}_{\rho}}$ and $\pi_{{\Delta}_{\varphi}}$ of $\mathbb{R}$;\\
$(iii)$ $\pi_{U}\cdot \pi_{{\Delta}_{\rho}}$ and $\pi_{V}\cdot \pi_{{\Delta}_{\varphi}}$ of $ G \times \mathbb{R}$.\\
Moreover, if $T_{\Phi}=W\abs{T_{\Phi}}$ denotes the polar decomposition of $T_{\Phi}$, then 
\begin{align*}
(i)^{\circ}\text{ }W\pi_{U}=\pi_{V}W,\text{ } (ii)^{\circ}\text{ }W\pi_{{\Delta}_{\rho}}=\pi_{{\Delta}_{\varphi}}W,\text{ }(iii)^{\circ}\text{ }W\pi_{U}\cdot\pi_{{\Delta}_{\rho}}=\pi_{V}\cdot\pi_{{\Delta}_{\varphi}}W.
\end{align*}
\end{proposition}

\begin{proof}
The final statement is a direct consequence of combining Eq. \eqref{adjointT_phi} above with a standard result in unitary representation theory $($cf. Proposition A.1.4 \cite{BHV}$)$. 
\end{proof}

\subsection{Ergodicity}
We now characterize ergodic systems in terms of joinings. A system $\mathfrak{N}=(N,\alpha,G,\rho)$ is \textbf{ergodic} if its fixed point algebra is trivial, i.e. $N^{G}=\{x\in N:\alpha_{g}(x)=x\}=\mathbb{C}1_{N}$. A system $\mathfrak{M}=(M,\beta, G ,\varphi)$ is an \textbf{identity system} if $\beta$ is the trivial action, i.e. $\beta_{g}=id_{M}$ for all
$g\in G $. Two systems $\mathfrak{N}$ and $\mathfrak{M}$ are said to be disjoint if 
$J_{s}(\mathfrak{N},\mathfrak{M})=\{\rho\otimes\varphi\}$, equivalently  
 $J_{m}(\mathfrak{N},\mathfrak{M})=\{\rho(\cdot)1_{M}\}$ by Theorem \ref{CharacterizeJoins} and \ref{CharacterizeReverse}. Note that by Theorem \ref{AdjointisEquivariant} it follows that if $J_{m}(\mathfrak{N},\mathfrak{M})=\{\rho(\cdot)1_{M}\}$, then
$J_{m}(\mathfrak{M},\mathfrak{N})=\{\varphi(\cdot)1_{N}\}$.

The next theorem is fundamental in classical joining theory.
Our setup provides a natural and short proof in full generality in the framework 
of von Neumann algebras.

\begin{theorem}\label{Ergodic}
A $W^{*}$--dynamical system is ergodic if and only if it is disjoint from every identity system.
\end{theorem}

\begin{proof}
Suppose $\mathfrak{N}$ is an ergodic system and $\mathfrak{M}$ is an identity
system. Suppose $\Phi\in J_{m}(\mathfrak{M},\mathfrak{N})$. Then for all $x\in M$ and all
$g\in G $ we have
\begin{align*}
T_{\Phi}(x\Omega_{\varphi})=T_{\Phi}(V_{g}x\Omega_{\varphi})=U_{g}T_{\Phi
}(x\Omega_{\varphi})\text{, from Eq. } \eqref{Eq: Crucial Equation}.
\end{align*}
Hence by ergodicity, $T_{\Phi}(x\Omega_{\varphi})\in\mathbb{C}\Omega_{\rho}$. But $T_{\Phi
}(x\Omega_{\varphi})=\Phi(x)\Omega_{\rho}$, hence $\Phi(x)\in
\mathbb{C}1_{N}$. However, since $\rho\circ\Phi=\varphi$, it follows that 
$\Phi(\cdot)=\varphi(\cdot)1_{N}$.

For the converse, suppose that $\mathfrak{N}$ is not ergodic. In this case,
the fixed point algebra $N^{ G }$ is nontrivial.
Note that $\sigma
_{t}^{\rho}\alpha_{g}=\alpha_{g}\sigma_{t}^{\rho}$ for all
$g\in G$  and $t\in\mathbb{R}$. It follows that $N^{ G }$ is
invariant under every $\sigma_{t}^{\rho}$, and therefore there is a unique faithful
normal conditional expectation $\mathbb{E}_{N^{ G }}$ from $N$ onto
$N^{ G }$ satisfying $\rho\circ\mathbb{E}_{N^{ G }}=\rho$ by a well--known Theorem of Takesaki \cite{Ta}. 
By uniqueness of $\mathbb{E}_{N^{G}}$ it follows
that $\mathbb{E}_{N^{ G }}\in J_{m}(\mathfrak{N},\mathfrak{M})$, where
$\mathfrak{M}=(N^{ G },\rho_{|N^{ G }},\alpha_{|N^{ G }}, G )$, 
which is a nontrivial identity system as $N^{ G }\neq \mathbb{C}1_{N}$. This
completes the argument.
\end{proof}

\begin{remark}
\emph{The first implication in Theorem \ref{Ergodic} can also be obtained by a standard
$2$--norm averaging technique which is a very popular tool in the theory of von Neumann algebras. 
If $\Phi\in J_{m}(\mathfrak{N},\mathfrak{M})$, where $\mathfrak{N}$ is ergodic and $\mathfrak{M}$ is an 
identity system, then for every $x$ in $N$, the unique element in $\overline{conv}^{||\cdot||_{\rho}}\{\alpha_{g}(x)\Omega_{\rho
}:g\in G \}$ of minimal norm corresponds to an element $\kappa(x)\in N^{ G }$; 
thus $\kappa(x)=\rho(x)1_{N}$ by ergodicity. 
Since $\mathfrak{M}$ is an identity system, $\Phi(\alpha
_{g}(x))=\Phi(x)$ for all $g\in G $ and this forces $\Phi(\kappa(x))=\Phi(x)$. It follows
that $\Phi(x)=\rho(x)1_{M}$ for all $x\in N$. Note that in all of these, the averaging argument is possible
since $\alpha_{g}$ preserves $\rho$ for all $g\in  G $.}
\end{remark}

\subsection{Primeness} Let $\mathfrak{M}$ be a system. Let $N\subset M$ be a von Neumann subalgebra that is globally invariant under the action of $G$ and as well as the modular automorphism group $\{\sigma_{t}^{\varphi}\}_{t\in \mathbb{R}}$. Then $\mathfrak{N}=(N, \varphi_{|N}, \beta_{|N}, G)$ is a subsystem of $\mathfrak{M}$. By \cite{Ta}, there is a faithful normal $\varphi$--preserving conditional expectation $\mathbb{E}_{N}$ from $M$ onto $N$. The system $\mathfrak{M}$ is said to be prime if it has no nontrivial subsystem.      

The following lemma is well--known to experts, but we prove it for convenience. 

\begin{lemma}\label{Fixed}

Let $M$ be a von Neumann algebra with a faithful normal state $\varphi$.
Let $\Phi:M\rightarrow M$ be a u.c.p. map such that $\varphi\circ \Phi=\varphi$. Then 
\begin{align*}
\{x\in M:\Phi(x)=x\}=\{x\in M: \Phi(xy)=x\Phi(y), \text{ } \Phi(yx)=\Phi(y)x \text{ }\forall y\in M\}
\end{align*}
is a von Neumann subalgebra of $M$. 
\end{lemma}

\begin{proof}
The hypothesis guarantees that $\Phi$ is normal. Note that the set on the right hand side of the statement is a $w^{*}$--closed $C^{\ast}$--subalgebra of $M$  which is contained in the set of fixed points of $\Phi$; while the set on the left hand side of the statement is a $w^{*}$--closed subspace of $M$. 

Let $x\in M$ be such that $\Phi(x)=x$. Then, 
$x^{*}x=\Phi(x)^{*}\Phi(x) \leq \Phi(x^{*}x)$. Thus, $\varphi(\Phi(x^{*}x)-x^{*}x)=0$ forces that $\Phi(x^{*}x)=x^{*}x=\Phi(x)^{*}\Phi(x)$. Consequently by Choi's theorem $($Theorem 3.1 \cite{Ch}$)$, one has 
$\Phi(xy)=\Phi(x)\Phi(y)=x\Phi(y)$ and $\Phi(yx)=\Phi(y)\Phi(x)=\Phi(y)x$ for all $y\in M$. 
\end{proof}

The following is a generalization of an unpublished result of K. Berg in the mid--seventies, a proof of which can be found in \cite{Go}. Note that the use of a u.c.p. map in the proof below reduces the technical intricacy. The final equivalence claim in the theorem is a direct consequence of the definition of ergodicity of positive maps on p. 7 of \cite{NSL}, which we mention for reference but do not discuss in detail here.  

\begin{theorem}\label{Berg}
A $W^{*}$--dynamical system $\mathfrak{M}$ is prime if and only if 
for each $id\neq\Phi\in J_{m}(\mathfrak{M},\mathfrak{M})$, $\Phi(x)=x$ implies $x\in \mathbb{C}1_{M}$. Equivalently, $\mathfrak{M}$ is prime if and only if each $id\neq\Phi\in J_{m}(\mathfrak{M},\mathfrak{M})$ is ergodic with respect to $\varphi$. 
\end{theorem}

\begin{proof}
First assume $\mathfrak{M}$ is prime. Suppose there exists $x\in M$ such that $\Phi(x)=x$ and $x\neq \varphi(x)1_{M}$. It follows from Lemma 
\ref{Fixed} that $M^{\Phi}=\{y\in M: \Phi(y)=y\}\neq \mathbb{C}1_{M}$ is a von Neumann subalgebra of $M$. Clearly,  $M^{\Phi}$ is $\sigma_{t}^{\varphi}$--invariant for all $t\in \mathbb{R}$ and $G$--invariant as well. Thus, $\mathfrak{N}=(M^{\Phi},\varphi_{|M^{\Phi}},\beta,G)$ is a subsystem of $\mathfrak{M}$. It follows that $M^{\Phi}=M$, equivalently $\Phi=id$. 

Conversely, suppose that $\mathfrak{N}$ is a subsystem of $\mathfrak{M}$ with $N\subseteq M$.  Then $\mathbb{E}_{N}\in J_{m}(\mathfrak{M},\mathfrak{M})$ by Theorem \ref{ConditionalExpectationisMarkov}. If $N\neq \mathbb{C}1_{M}$, then there exists $x\in M$ such $x\not\in \mathbb{C}1_{M}$ 
and $\mathbb{E}_{N}(x)=x$. But by the hypothesis, this is only possible if $\mathbb{E}_{N}=id$, whence $N=M$. Thus, $\mathfrak{M}$ is prime. 
\end{proof}

\subsection{Weak Mixing}

For general facts about weakly mixing actions, we refer the reader to the survey article \cite{BeGo} and to \cite{KLP}. 
Weak mixing describes a certain asymptotic independence of the system: the tendency of a centered observable under a subsequence of iterates of symmetries to become approximately orthogonal to any other observable \cite{BeRo}. This phenomenon for $\mathbb{Z}$--actions is also characterized by the abundance of
weakly wandering vectors, reflecting the asymptotic independence \cite{Kr}. 

\begin{definition}\label{WM}
$(i)$ A system $\mathfrak{N}$ is said to be \textbf{weakly mixing} if for every
$\varepsilon>0$ and every finite subset $F\subseteq N$ such that $\rho(x)=0$ for
all $x\in F$, there exists $g\in G $ such that $|\rho(\alpha
_{g}(x)y)|<\varepsilon$ for all $x,y\in F$.\\
$(ii)$ A system $\mathfrak{N}$ is \textbf{compact} if the orbit $\alpha_{ G 
}(x)\Omega_{\rho}$ of $x\Omega_{\rho}$ is totally bounded in $\norm{\cdot}_{\rho}$ for all $x\in
N$.
\end{definition}

In direct analogy with the classical case (cf. \cite{KeLi}), the unitary representation $g\mapsto U_{g}$ of $ G $ on $L^{2}(N,\rho)\ominus\mathbb{C}\Omega_{\rho}$ associated to a system $\mathfrak{N}=(N,\rho,\alpha,G)$ is weakly mixing if and only if it has no finite--dimensional subrepresentation.  In particular, then, the latter condition on the associated representation implies that the original system $\mathfrak{N}$ is weakly mixing.\footnote{We point out that the definition of a weakly mixing
system appearing in Def. 2.5 of \cite{Du2} is appropriate only in the case when the acting group is abelian, since a non-abelian group may admit irreducible, finite-dimensional (hence, non--weakly mixing) representations of dimension greater than one.  }

\begin{definition}\label{EigenFunction}
Let $\mathfrak{N}$ be a system. An element $0\neq x\in N$ is an
\textbf{eigenfunction} for $\pi_{U}$ if there exists a continuous character $\chi
: G \rightarrow S^{1}$ such that $U_{g}(x\Omega_{\rho})=\langle
g,\chi\rangle x\Omega_{\rho}$ for all $g\in G $. A nonzero vector $\xi \in L^{2}(N,\rho)$ such that there exists a continuous character $\chi
: G \rightarrow S^{1}$ satisfying $U_{g}\xi=\langle
g,\chi\rangle \xi$ for all $g\in G $ will be called an \textbf{eigenvector} for $\pi_{U}$.

\end{definition}
Since $U_{g}\Omega_{\rho}=\Omega_{\rho}$ for all $g\in G$, it follows that if $x\in N$ is an eigenfunction and the associated character is non trivial then $\rho(x)=0$. 
Moreover, both $x^{*}x$ and $xx^{*}$ lie in $N^{ G }$, and therefore if the system is ergodic then $x$ is a scalar multiple of a unitary. It is a standard fact that the von Neumann subalgebra $N_{0}$ of $N$ 
generated by  eigenfunctions of $\pi_{U}$ is a compact subsystem 
of $\mathfrak{N}$ $($which is nontrivial when $\mathfrak{N}$ is not weakly mixing 
and the acting group is abelian, cf. Theorem 2.1 \cite{NeSt}$)$. Indeed, as
$\alpha_{g}\sigma_{\rho}^{t}=\sigma_{\rho}^{t}\alpha_{g}$ for $g\in  G $ and $t\in\mathbb{R}$, $N_{0}$ is invariant under the action of the modular automorphism group, and thus by \cite{Ta}
there is a unique $\rho$--preserving faithful normal conditional expectation $\mathbb{E}_{N_{0}}$ onto $N_{0}$ which is 
$ G $--equivariant by the uniqueness of $\mathbb{E}_{N_{0}}$. Thus, the adjoint of the map 
$\mathbb{E}_{N_{0}}$ $($which is the inclusion map by Theorem \ref{ConditionalExpectationisMarkov} above$)$ is also $ G $--equivariant by Theorem 
\ref{AdjointisEquivariant} and satisfies all of the properties of a subsystem $($see Definition \ref{Subsystem} and
Theorem \ref{ConditionalExpectationisMarkov} for details$)$. The total boundedness of orbits 
follows from straightfoward density arguments.

Note that if $\mathfrak{N}$ is  a compact system, then $vN(\{U_{g}:g\in  G \})\subseteq \textbf{B}(L^{2}(N,\rho))$ 
is atomic and is in fact a direct sum of matrix algebras. Indeed, first note that the unit ball of $\textbf{B}(L^{2}(N,\rho))$ 
equipped with the strong operator topology is a complete metric space, since $L^{2}(N,\rho)$ is separable. Using the separability of $L^{2}(N,\rho)$, a routine diagonalization argument yields that the strong--operator closure of $\{U_{g}:g\in  G \}$ is sequentially compact. 
It follows that $K=\overline{\{U_{g}:g\in  G \}}^{s.o.t.}$ is a compact group of unitaries. By the Peter--Weyl theorem,
the identity representation of $K$ on $L^{2}(N,\rho)$ is a direct sum of irreducible subrepresentations,
each of which is finite--dimensional. On the other hand, note that if $\mathfrak{N}$ were weakly mixing, 
then the representation $\pi_{U}$ could not have any finite--dimensional direct summand on the orthocomplement of 
$\mathbb{C}\Omega_{\rho}$. In fact, we have the following result.

\begin{theorem}\label{WmDisjointCompact}
A weakly mixing system is disjoint from any compact system.
\end{theorem}

\begin{proof}
By Theorem \ref{CharacterizeJoins} and \ref{CharacterizeReverse}, it suffices to
work with u.c.p. maps. Let $\mathfrak{M}$ be a weakly mixing system and suppose $\mathfrak{N}$ is compact. If $\Phi\in J_{m}(\mathfrak{N},\mathfrak{M})$, then by $(i)^{\circ}$ of
Proposition \ref{Prop:Intertwine}, it follows that 
\begin{align*}
W\pi_{U}=\pi_{V}W, 
\end{align*}
where $W$ is the polar part of $T_{\Phi}$. Note that $Ker(W)^{\perp}=Ker(T_{\Phi})^{\perp}$ and $Ran(W)
=\overline{Ran(T_{\Phi})}$ are reducing subspaces of $\pi_{U}$ and $\pi_{V}$ respectively. 

If $Ker(W)\neq L^{2}(N,\rho)\ominus\mathbb{C}\Omega_{\rho}$ then there is a $\xi\notin \mathbb{C}\Omega_{\rho}$ orthogonal to $Ker(W)$.  By compactness of $\mathfrak{N}$ and the remarks above, 
 $\overline{\pi_{U}(G)\xi}$ decomposes as a direct sum of finite dimensional reducing subspaces for $\pi_{U}$. Therefore the nontrivial subspace
$W\overline{\pi_{U}(G)\xi}=\overline{\pi_{V}(G)W\xi}$ of $L^2(M,\varphi)$ decomposes as a direct sum of finite-- dimensional reducing subspaces for $\pi_{V}$, contradicting the fact that $\mathfrak{M}$ was weakly mixing.
Consequently, $\Phi(\cdot)=\rho(\cdot)1_{M}$.
\end{proof}

\subsection{Polar Decomposition}\label{PD}

In order to characterize weakly mixing systems via joinings, we need some preparatory results of independent interest in the theory of von Neumann algebras. The main result in this subsection is Theorem \ref{InvaraintSubspaceInsideCentralizer}, which says that finite--dimensional invariant subspaces of the canonical unitary representation associated to a state--preserving ergodic action of a separable locally compact group on a von Neumann algebra $M$ lie inside the subspace $M^{\varphi}\Omega_{\varphi}$ generated by the centralizer of the state. We exploit a certain rigidity in the polar decomposition of operators affiliated to a von Neumann algebra together with a rigidity of the spectral decomposition of such operators. Together with the powerful main result of \cite{HLS}, this locates the finite--dimensional invariant subspaces inside the centralizer.   
We need some preparation, details of which can be found in \cite{Str}. 

Note that for $\zeta\in \mathfrak{D}(S_{\varphi})$, the operator $L_{\zeta}^{0}$ with $\mathfrak{D}(L_{\zeta}^{0})=M^{\prime}\Omega_{\varphi}$ and (densely) defined by $L_{\zeta}^{0}(x^{\prime}\Omega_{\varphi})=x^{\prime}\zeta$ for all $x^{\prime}\in M^{\prime}$, is closable and affiliated to $M$. Let $L_{\zeta}$ denote the closure of $L_{\zeta}^{0}$, and let $L_{\zeta}=v|L_{\zeta}|$ denote the polar decomposition of $L_{\zeta}$. Thus $|L_{\zeta}|$ is positive, self--adjoint and affiliated to $M$ and $v\in M$. 

Let $P_{S_{\varphi}}= \{T\Omega_{\varphi}:  T \text{ is positive, self--adjoint and affiliated to } M \text{ and } \Omega_{\varphi}\in \mathfrak{D}(T)\}$. Note that $\Omega_{\varphi}\in \mathfrak{D}(L_{\zeta})=\mathfrak{D}(\abs{L_{\zeta}})$. Then, if  $\eta=\abs{L_{\zeta}}\Omega_{\varphi}$, then $\eta\in P_{S_{\varphi}}$ and $\zeta = v \eta$. Note also that $\eta\in P_{S_{\varphi}}=\overline{M_{+}\Omega_{\varphi}}^{\norm{\cdot}_{\varphi}}$. Thus, $L_{\eta}^{0}$ $($defined as above$)$ is positive $($i.e., $\langle x^{\prime}\eta,x^{\prime}\Omega_{\varphi}\rangle_{\varphi}\geq 0$, for all $x^{\prime}\in M^{\prime})$. Clearly, $L_{\zeta}^{0}= v L_{\eta}^{0}$, thus $\mathfrak{D}(L_{\zeta}^{0})=\mathfrak{D}(L_{\eta}^{0})$. Then $v^{*} L_{\zeta}^{0}$ is closable with $\overline{v^{*} L_{\zeta}^{0}}=v^{*}L_{\zeta}$. So $L_{\eta}^{0}$ is closable and $L_{\zeta}=vL_{\eta}$, where $L_{\eta}$ denotes the closure of $L_{\eta}^{0}$. This shows that $L_{\eta}=v^{*}L_{\zeta}=\abs{L_{\zeta}}$, so $L_{\eta}$ is positive and self--adjoint. Finally, $L_{\zeta}=v L_{v^{*}\zeta}$ denotes the polar decomposition of $L_{\zeta}$. 

Let $n\in \mathbb{N}$, and let $\widetilde{M}=M\otimes M_{n}(\mathbb{C})$. 
Consider the faithful normal state $\widetilde{\varphi}=\varphi\otimes tr$ on $\widetilde{M}$, where $tr$ denotes the normalized trace on $M_{n}(\mathbb{C})$. The GNS Hilbert space $L^{2}(\widetilde{M},\widetilde{\varphi})$ of $\widetilde{M}$ with respect to  $\widetilde{\varphi}$ is $L^{2}(M,\varphi)\otimes M_{n}(\mathbb{C})$ in which $\Omega_{\widetilde{\varphi}}$ will be the standard vacuum vector. Let $1_{n}$ be the identity operator on the Hilbert space $M_{n}(\mathbb{C})$. Vectors in $L^{2}(M,\varphi)\otimes M_{n}(\mathbb{C})$ are viewed as $n\times n$ matrices with elements from $L^{2}(M,\varphi)$ and a generic element in $L^{2}(M,\varphi)\otimes M_{n}(\mathbb{C})$ will be written as $[\zeta_{ij}]$ with $\zeta_{ij}\in L^{2}(M,\varphi)$ for $1\leq i,j\leq n$. The obvious analogous notational convention is adopted for elements in $\widetilde{M}$. 

The associated operators in the modular theory of $\widetilde{M}$ with respect to 
$\widetilde{\varphi}$ are given by $S_{\widetilde{\varphi}}=S_{\varphi}\otimes J_{tr}$, $J_{\widetilde{\varphi}}=J_{\varphi}\otimes J_{tr}$, $\Delta_{\widetilde{\varphi}}=\Delta_{\varphi}\otimes 1_{n}$ and the modular automorphisms are given by $\sigma_{t}^{\widetilde{\varphi}} =\sigma_{t}^{\varphi}\otimes 1$ for all $t\in\mathbb{R}$. Thus, $\Delta^{it}_{\widetilde{\varphi}}[\zeta_{ij}] = [\Delta^{it}_{\varphi}\zeta_{ij}]$, for all $t\in \mathbb{R}$ and all $\zeta_{ij}\in L^{2}(M,\varphi)$ for $1\leq i,j\leq n$. Also note that, since $S_{\widetilde{\varphi}}=S_{\varphi}\otimes J_{tr}$, it follows that $[\zeta_{ij}] \in \mathfrak{D}(S_{\widetilde{\varphi}})$ if and only if $\zeta_{ij}\in \mathfrak{D}(S_{\varphi})$ for all $1\leq i,j\leq n$. 

Consequently, if $[\zeta_{ij}] \in \mathfrak{D}(S_{\widetilde{\varphi}})$, then the densely defined operator $L_{[\zeta_{ij}]}^{0}$, defined as before, is closable and its closure $L_{[\zeta_{ij}]}$ admits a polar decomposition $L_{[\zeta_{ij}]}= [u_{ij}]|L_{[\zeta_{ij}]}|$, where $[u_{ij}]\in \widetilde{M}$ is a partial isometry and $|L_{[\zeta_{ij}]}|$ is positve, self--adjoint and affiliated to $\widetilde{M}$. Thus, as before, we also have 
\begin{align*}
&[\zeta_{ij}]=[u_{ij}][\eta_{ij}], \text{ where }[\eta_{ij}]=|L_{[\zeta_{ij}]}|\Omega_{\widetilde{\varphi}} \text{ and }\\
&L_{[\zeta_{ij}]} = [u_{ij}] L_{[u_{ij}]^{*}[\zeta_{ij}]}.
\end{align*}

The following lemma, which is crucial for our purposes, is probably well--known to experts but we lack a reference, and therefore include a proof for the sake of completeness.

\begin{lemma}\label{polardecomposition}
Let $M$ be a von Neumann algebra equipped with a faithful normal state $\varphi$. Let $\alpha\in Aut(M,\varphi)$. Let $V_{\alpha}\in \mathbf{B}(L^{2}(M,\varphi))$ denote the unique unitary such that $(Ad\text{ } V_{\alpha})(x)=\alpha(x)$ for all $x\in M$. Let $\zeta \in \mathfrak{D}(S_{\varphi})$. Then,
$V_{\alpha}\zeta\in  \mathfrak{D}(S_{\varphi})$, and if $L_{\zeta}=vL_{\eta}$ denotes the polar decomposition of $L_{\zeta}$, then 
\begin{align*}
L_{V_{\alpha}\zeta}=\alpha(v) L_{V_{\alpha}\eta}
\end{align*}
is the polar decomposition of $L_{V_{\alpha}\zeta}$. 
\end{lemma}

\begin{proof}
 Let $\mathcal{D}_{\frac{1}{2}}=\{z\in \mathbb{C}: 0\leq \mathfrak{R}(z)\leq \frac{1}{2}\}$. Note that $\mathfrak{D}(S_{\varphi})=\mathfrak{D}(\Delta_{\varphi}^\frac{1}{2})$. Also note that $V_{\alpha}$ commutes with $\Delta_{\varphi}^{it}$ for all $t\in \mathbb{R}$. Since $\zeta\in  \mathfrak{D}(S_{\varphi})$, let $F: \mathcal{D}_{\frac{1}{2}}\rightarrow L^{2}(M,\varphi)$ be a function continuous on $\mathcal{D}_{\frac{1}{2}}$ and analytic in its interior that satisfies $F(it)=\Delta_{\varphi}^{it}\zeta$ for all $t\in \mathbb{R}$. Then $V_{\alpha}\circ F:\mathcal{D}_{\frac{1}{2}}\rightarrow L^{2}(M,\varphi)$ is continuous and analytic in the interior of $\mathcal{D}_{\frac{1}{2}}$ such that $(V_{\alpha}\circ F)(it)=\Delta_{\varphi}^{it}V_{\alpha}\zeta$, for all $t\in \mathbb{R}$. Thus, $V_{\alpha}\zeta\in \mathfrak{D}(S_{\varphi})$.

By the previous discussion, $L_{V_{\alpha}\zeta}^{0}$ is closable. 
Since $P_{S_{\varphi}}=\overline{M_{+}\Omega_{\varphi}}^{\norm{\cdot}_{\varphi}}$, it follows that $V_{\alpha}\eta\in P_{S_{\varphi}}$. Consequently, $L_{V_{\alpha}\eta}^{0}$ is closable and $\alpha(v)L_{V_{\alpha}\eta}=L_{V_{\alpha}\zeta}$. Indeed, since 
$\zeta=v\eta$ and $\eta\in P_{S_{\varphi}}$, a straightforward approximation argument yields $V_{\alpha}\zeta = \alpha(v) V_{\alpha}{\eta}$. So, $L_{V_{\alpha}\zeta}^{0}=\alpha(v)L_{V_{\alpha}\eta}^{0}$. Observe that $V_{\alpha}L_{\eta}V_{\alpha}^{*}$ is densely defined. Note that from the discussion in \S2, it follows that for $x^{\prime}\in M^{\prime}$ we have
\begin{align*}
V_{\alpha}L_{\eta}^{0}V_{\alpha}^{*}(x^{\prime}\Omega_{\varphi})&= V_{\alpha}L_{\eta} ^{0}((\alpha^{-1})^{op}(x^{\prime}))\Omega_{\varphi}\\&=V_{\alpha}((\alpha^{-1})^{op}(x^{\prime}))\eta =x^{\prime} V_{\alpha}\eta = L_{V_{\alpha}\eta}^{0}(x^{\prime}\Omega_{\varphi}).
\end{align*}
Hence, $V_{\alpha}L_{\eta}^{0}V_{\alpha}^{*}=L_{V_{\alpha}\eta}^{0}$, but the left hand side of the previous equality is closable. Thus, $L_{V_{\alpha}\eta}^{0}$ is closable and, denoting its closure by $L_{V_{\alpha}\eta}$, it follows that $L_{V_{\alpha}\eta}=\overline{V_{\alpha}L_{\eta}^{0}V_{\alpha}^{*}}=V_{\alpha}L_{\eta}V_{\alpha}^{*}$. Consequently, $L_{V_{\alpha}\zeta}=\alpha(v)L_{V_{\alpha}\eta}$ and $L_{V_{\alpha}\eta}$ is positive, self--adjoint $($as $L_{V_{\alpha}\eta}=V_{\alpha}L_{\eta}V_{\alpha}^{*})$ and affiliated to $M$. 

Clearly, $M^{\prime}\Omega_{\varphi}$ is a core for $L_{V_{\alpha}\eta}$. Thus, the left support of $L_{V_{\alpha}\eta}$ $($which is its right support as well$)$,  is the projection onto the subspace $\overline{L_{V_{\alpha}\eta}M^{\prime}\Omega_{\varphi}}^{\norm{\cdot}_{\varphi}}$. But a simple calculation shows that 
$\overline{L_{V_{\alpha}\eta}M^{\prime}\Omega_{\varphi}}^{\norm{\cdot}_{\varphi}}=
\overline{V_{\alpha}M^{\prime}\eta}^{\norm{\cdot}_{\varphi}}=V_{\alpha}\overline{M^{\prime}\eta}^{\norm{\cdot}_{\varphi}}$. 
But the projection onto $\overline{M^{\prime}\eta}^{\norm{\cdot}_{\varphi}}$ is the left $($as well as the right$)$ support of $L_{\eta}$,
which is $v^{*}v$. It follows that each of the left and right support of $L_{V_{\alpha}\eta}$ is $\alpha(v^{*}v)$. The argument is complete upon invoking the uniqueness of the polar decomposition.
\end{proof}

\subsubsection{\textbf{Structure of Finite--Dimensional Invariant Subspaces}}

In this section, we exhibit rigidity in the structure of finite--dimensional invariant subspaces of an 
ergodic action. The proof of the main result of this section requires two stages. The first stage handles eigenvectors using hard analysis, and the second handles higher dimensional invariant subspaces by a sort of amplification. 

Let $z_0 \in \mathbb{C}$ and consider the strip $I_{z_0}$ defined as $\{z\in\mathbb{C}|0\le \Re z \le \Re z_0\}$  when $\Re z_0\ge 0$ and as $\{z\in\mathbb{C}|0\le \Re z_0 \le \Re z\}$  when $\Re z_0<0$. Recall the standard fact (cf. 9.20 and 9.21 of \cite{StZs}), that for $(\xi,\eta)\in$ Graph$(\Delta^{z_0}_\varphi)$  if and only if the map $i\mathbb{R}\ni is \mapsto \Delta_{\varphi}^{is}\xi\in L^{2}(M,\varphi)$ has a continuous extension $f$ to $I_{z_0}$ that analytic on this strip's interior and satisfies $f(z_0)=\eta$. 
We call a vector $\xi\in L^{2}(M,\varphi)$ an \textbf{analytic vector} if the map $i\mathbb{R}\ni is \mapsto \Delta_{\varphi}^{is}\xi\in L^{2}(M,\varphi)$
has a weakly entire extension that is continuous on $I_{z_0}$ for every $z_0 \in \mathbb{C}$.
It is also a standard result (see, e.g. 10.16 of \cite{StZs}) that for $\xi\in L^{2}(M,\varphi)$ and $r>0$, 
\begin{align}\label{Analytic}
\xi_{r} = \sqrt{\frac{r}{\pi}}\int_{-\infty}^{\infty}e^{-rt^{2}}\Delta_{\varphi}^{it}\xi dt\rightarrow \xi \text{ as } r\rightarrow\infty \text{ in } \norm{\cdot}_{\varphi},
\end{align}
with each $\xi_{r}$ an analytic vector. We will refer to such a net of vectors $(\xi_r)_{r>0}$ as an \textbf{analytic approximation of} $\xi$. 
The proof of the above result uses the fact that Gaussian distributions with shrinking variances converge weakly to the Dirac measure at $0$. For each $r>0$, the function $F_{r}: \mathbb{C}\rightarrow L^{2}(M,\varphi)$ defined by $F_{r}(z)= \sqrt{\frac{r}{\pi}}\int_{-\infty}^{\infty}e^{-r(t+iz)^{2}}\Delta_{\varphi}^{it}\xi dt$ is weakly entire and continuous on $I_{z_0}$ for all $z_0\in \mathbb{C}$, and extends $i\mathbb{R}\ni is \mapsto \Delta_{\varphi}^{is}\xi_{r}\in L^{2}(M,\varphi)$.  Thus, for all $z\in \mathbb{C}$ we have $\xi_{r}\in \mathfrak{D}(\Delta_{\varphi}^{z})$ and, in particular, 
$\xi_{r}\in\mathfrak{D}(\Delta_{\varphi}^{\frac{1}{2}})$. Consequently, $\xi_{r}\in \mathfrak{D}(S_{\varphi})$. 

The following two theorems locate all finite-dimensional invariant subspaces for the representation $g \mapsto V_g$ inside the centralizer of the underlying von Neumann algebra. The first, Theorem \ref{eigencase}, compares with Theorem 2.5 of \cite{St}. The point of view of its proof is then employed, in concert with the main result of \cite{HLS}, to prove Theorem \ref{InvaraintSubspaceInsideCentralizer}.

\begin{theorem}\label{eigencase}
Let $G$ act ergodically on $M$ preserving $\varphi$. Let $\xi\in L^{2}(M,\varphi)$ be a unit vector such that $V_{g}\xi = \langle g,\chi\rangle \xi$ for all $g\in G$, where $\chi$ is a $($continuous$)$ character of $G$. Then $\xi=u\Omega_{\varphi}$, where $u\in M^{\varphi}$ is a unitary. Moreover, if $\xi\in L^{2}(M,\varphi)\ominus \mathbb{C}\Omega_{\varphi}$ then $\varphi(u)=0$.
\end{theorem}

\begin{proof}
Let $\{\xi_{r}:r>0\}$, with $\xi_{r}$ defined in Eq. \eqref{Analytic}, be an analytic approximation of $\xi$. Then there exists $r_{0}$ such that $\xi_{r}\neq 0$ for all $r\geq r_{0}$. Fix $r\geq r_{0}$. Note that $V_{g}\Delta_{\varphi}^{it}=\Delta_{\varphi}^{it}V_{g}$ for all $g\in G$ and for all $t\in\mathbb{R}$. So $V_{g}\xi_{r}=\langle g,\chi\rangle \xi_{r}$ for all $g\in G$. Since $\xi_{r}\in \mathfrak{D}(S_{\varphi})$, it then follows that $V_g \xi_{r} \in \mathfrak{D}(S_{\varphi})$ as well.

Denote by $L_{\xi_{r}}=v_{r}L_{\eta_{r}}$ the polar decomposition of  $L_{\xi_{r}}$. By Lemma \ref{polardecomposition}, 
$L_{V_{g}\xi_{r}}=\beta_{g}(v_{r}) L_{V_{g}\eta_{r}}$ 
is the polar decomposition of $L_{V_{g}\xi_{r}}$, for each $g\in G$.   On the other hand, for any such $g$ we have $L_{V_g \xi_r} = \langle g,\chi\rangle L_{\xi_r} = \langle g, \chi \rangle v_r L_{\eta_r},$
so for all $g \in G$, 
\begin{align*}
L_{V_{g}\xi_{r}}=\langle g,\chi\rangle v_{r} L_{\eta_{r}}=\beta_{g}(v_{r}) L_{V_{g}\eta_{r}}. 
\end{align*}
Hence, by uniqueness of the polar decomposition, it follows that for all $g\in G$ one has 
\begin{align*}
\beta_{g}(v_{r})=\langle g,\chi\rangle v_{r} \text{ and } V_{g}\eta_{r}=\eta_{r}.
\end{align*}
Consequently, $\beta_{g}(v_{r}^{*}v_{r})=v_{r}^{*}v_{r}$ for all $g\in G$, which by ergodicity implies $v_{r}^{*}v_{r}=c1_{M}$, where $c>0$. But $v_{r}$ is a nonzero partial isometry, thus $\norm{v_{r}}=1$. 	So $c=\norm{v_{r}}^{2}=1$, and therefore $v_{r}^{*}v_{r}=1_{M}$. We similarly obtain $v_{r}v_{r}^{*}=1_{M}$ and therefore $v_{r}\in M$ is unitary. 

Now, by the spectral theorem, write $L_{\eta_{r}} =\int_{0}^{\infty}\lambda de_{\lambda}$. Note that $e_{\lambda}\in M$ for all $\lambda\geq 0$. Since $V_{g}\eta_{r}=\eta_{r}$, we have that $L_{\eta_{r}}=\int_{0}^{\infty}\lambda d\beta_{g}(e_{\lambda})$ for all $g\in G$. Since the resolution of the identity in the spectral theorem is unique $($see, e.g., Theorem 5.6.15 of \cite{KRI}$)$, it follows that 
$\beta_{g}(e_{\lambda})=e_{\lambda}$  for all $g\in G$, and ergodicity implies  that $e_{\lambda}$ is trivial for all $\lambda$. Consequently, the support of $L_{\eta_{r}}$ is $1_M$, and therefore $\eta_{r}=b_r\Omega_{\varphi}$ for some real number $b_r>0$. 

Now, $\sigma_{t}^{\varphi}\beta_{g}=\beta_{g}\sigma_{t}^{\varphi}$ implies that $\beta_{g}(\sigma_{t}^{\varphi}(v_{r}))=\langle g,\chi\rangle \sigma_{t}^{\varphi}(v_{r})$ for all $t\in \mathbb{R}$ and all $g\in G$. Thus, for all $g\in G$ and $t\in \mathbb{R}$ 
\begin{align*}
\beta_{g}(v_{r}^{*}\sigma_{t}^{\varphi}(v_{r}))=v_{r}^{*}\sigma_{t}^{\varphi}(v_{r}).
\end{align*}
Hence, by ergodicity of the action we have, $\sigma_{t}^{\varphi}(v_{r})=c_{t}v_{r}$,  where $\abs{c_{t}}=1$ for all $t\in \mathbb{R}$. Since $\sigma^{\varphi}_{t+s}=\sigma^{\varphi}_{t}\circ \sigma^{\varphi}_{s}$ for all $t,s\in \mathbb{R}$, we have $c_{t+s}=c_{t}c_{s}$, and $\{c_{t}\}_{t\in \mathbb{R}}$ is a subgroup of the unit circle. Thus, there is an $\omega\in \mathbb{R}$ such that $c_{t}=e^{i\omega t}$ for all $t\in \mathbb{R}$, i.e.   $\sigma_{t}^{\varphi}(v_{r})=e^{i\omega t}v_{r}$ for all $t\in \mathbb{R}$. The function $i\mathbb{R}\ni it\mapsto \Delta_{\varphi}^{it}(v_{r}\Omega_{\varphi}) \in L^{2}(M,\varphi)$ has an entire extension $\mathbb{C}\ni z\mapsto e^{\omega z} v_{r}\Omega_{\varphi}\in L^{2}(M,\varphi)$. Thus, 
$\Delta_{\varphi}^{\frac{1}{2}}v_{r}\Omega_{\varphi}=e^{\frac{\omega}{2}} v_{r}\Omega_{\varphi}$. Since $S_{\varphi}=J_{\varphi}\Delta_{\varphi}^{\frac{1}{2}}$ and $S_{\varphi}$ is the closure of $S_{0,\varphi}$, we obtain that 
$v_{r}^{*}\Omega_{\varphi}= J_{\varphi}\Delta_{\varphi}^{\frac{1}{2}}v_{r}\Omega_{\varphi}=e^{\frac{\omega}{2}}J_{\varphi}v_{r}\Omega_{\varphi}$, i.e., $J_{\varphi}v_{r}\Omega_{\varphi}= e^{-\frac{\omega}{2}}v_{r}^{*}\Omega_{\varphi}$. As $J_{\varphi}$ is anti--unitary and $v_{r}$ is unitary, we see that $\omega=0$ upon applying $\norm{\cdot}_{\varphi}$ to the previous expression. Consequently, $v_{r}\in M^{\varphi}$, so that $\xi_{r}\in  L^{2}(M^{\varphi},\varphi)$. Since this is true for all $r\geq r_{0}$ and $\xi_{r}\rightarrow \xi$ in $\norm{\cdot}_{\varphi}$, we find that  $\xi \in L^{2}(M^{\varphi},\varphi)$. 

Note that 
$G\ni g\mapsto \beta_{g}$ implements an action on $M^{\varphi}$ preserving $\varphi$.  Restricting our attention now to $ L^{2}(M^{\varphi},\varphi)$, 
if we let $L_{\xi}=uL_{\eta}$  denote the polar decomposition of $\xi$, then  $u\in M^{\varphi}$ and $\eta\in L^{2}(M^{\varphi},\varphi)$.  
Since $\xi\in L^{2}(M^{\varphi},\varphi)$, Lemma \ref{polardecomposition} and the arguments used on the $\xi_{r}$ apply to $\xi$ as well and we have 
\begin{align*}
\beta_{g}(u)=\langle g,\chi\rangle u \text{ and } V_{g}\eta=\eta \text{ for all }g\in G. 
\end{align*}
By the same argument as above, with the additional fact that $\xi$ was a unit vector, $\eta=\Omega_{\varphi}$ and 
$u\in \mathcal{U}(M^{\varphi})$. It follows that $\xi = u\Omega_{\varphi}$. 

If $\xi=u\Omega_{\varphi}\in L^{2}(M,\varphi)\ominus \mathbb{C}\Omega_{\varphi}$ then $\varphi(u)=\langle u\Omega_{\varphi},\Omega_{\varphi}\rangle=0$. This completes the proof.
\end{proof}

The next theorem is a generalization of Theorem \ref{eigencase}. The idea in its proof is to find an appropriate replacement of the character appearing in Theorem \ref{eigencase}. We need to amplify the dynamics appropriately to obtain the result. To keep notation simple, we will use the same symbols $\widetilde{M}$ and so on as discussed in the beginning of \S\ref{PD} to denote the amplifications. This will be clear from the context and will cause no confusion.

\begin{theorem}\label{InvaraintSubspaceInsideCentralizer}
Let $G$ act ergodically on $M$ preserving the state $\varphi$. Let $\mathcal{K}\subseteq L^{2}(M,\varphi)$ be a nonzero finite--dimensional invariant subspace of the associated representation $\pi_{V}$ of $G$ on $L^{2}(M,\varphi)$. Then $\mathcal{K}\subseteq M^{\varphi}\Omega_{\varphi}$. 
\end{theorem}

\begin{proof}
We first show that there exists a finite--dimensional invariant subspace of analytic vectors. Let $dim (\mathcal{K})=m$ and let $(\xi_{1},\xi_{2},\cdots, \xi_{m})$ be an orthonormal basis of $\mathcal{K}$. Following Eq. \eqref{Analytic}, let $\xi_{i,r}$, $r>0$, be analytic approximations of $\xi_{i}$ for all $1\leq i\leq m$. 
There exists $r_{0}>0$ such that $\xi_{i,r}\neq 0$ for all $r>r_{0}$ and for all $1\leq i\leq m$. Fix $r>r_{0}$. As $V_{g}\Delta_{\varphi}^{it}=\Delta_{\varphi}^{it}V_{g}$ for all $t\in \mathbb{R}$ and all $g\in G$, the space 
$\mathcal{K}_{r}=\text{span }\{\xi_{i,r}:1\leq i\leq m\}$ is $G$--invariant. Let $(\zeta_{1},\zeta_{2},\cdots, \zeta_{n})$ be an orthonormal basis of $\mathcal{K}_{r}$. Note that $n\leq m$.

Consider $\widetilde{M}=M\otimes M_{n}(\mathbb{C})$, equipped with the faithful normal state $\widetilde{\varphi}= \varphi\otimes tr$, where $tr$ is the normalized trace on $M_{n}(\mathbb{C})$. Borrowing notation and the facts associated to the amplification from \S\ref{PD}, let $\widetilde{M}$ act on the left of $L^{2}(\widetilde{M},\widetilde{\varphi})$ in the GNS representation. Let $G$ act on 
$\widetilde{M}$ via $\widetilde{\beta}_{g}=\beta_{g}\otimes 1$, $g\in G$. As the $G$--action on $M$ is ergodic, it follows that the fixed point algebra $\widetilde{M}^{G}$ of the $G$ action on $\widetilde{M}$ is given by $\widetilde{M}^{G}=1\otimes M_{n}(\mathbb{C})$. Note that $\widetilde{\beta}_{g}$ is implemented by the unitary $\widetilde{V}_{g}=V_{g}\otimes 1_{n}$ for each $g\in G$. Let 
\begin{align*}
\zeta=\begin{pmatrix}\zeta_{1}&0&0\text{ }\cdots &0\\
                                             \zeta_{2}&0&0\text{ }\cdots &0\\
                                             \vdots&\vdots&\vdots\text{ }\cdots &\vdots\\                                           
                                             \zeta_{n}&0&0\text{ }\cdots &0\\
                       \end{pmatrix}\in L^{2}(\widetilde{M},\widetilde{\varphi}).
\end{align*}

Note that $\zeta\in \mathfrak{D}(S_{\widetilde{\varphi}})$, since $\zeta_{j}$ is analytic for all $1\leq j\leq n$. To see this, note that each $\zeta_{j}$ is a finite linear combination of the analytic vectors $\xi_{k,r}$ and use the fact that linear combinations of entire functions are entire to obtain the desired entire extension.

Let $L_{\zeta} = [u_{ij}]L_{[\eta_{ij}]}$ denote the polar decomposition of $L_\zeta$. Then by Lemma \ref{polardecomposition}, it follows that for all $g\in G$,
\begin{align}\label{polardecompofzeta}
L_{\widetilde{V}_{g}\zeta} = \widetilde{\beta}_{g}([u_{ij}])L_{\widetilde{V}_{g}[\eta_{ij}]}
\end{align}
is the polar decomposition of $L_{\widetilde{V}_{g}\zeta}$. 

Since $\mathcal{K}_{r}$ is an invariant subspace of the $G$--action, for every $g\in G$ we have that $(V_{g}\zeta_{i})_{i=1}^{n}$ is an orthonormal basis of $\mathcal{K}_{r}$. Thus, for $g\in G$, there exists a unitary $v_{g}\in \mathcal{U}(M_{n}(\mathbb{C}))$ such that $(V_{g}\zeta_{1}, V_{g}\zeta_{2},\cdots, V_{g}\zeta_{n})^{T}=(1\otimes v_{g})(\zeta_{1},\zeta_{2},\cdots,\zeta_{n})^{T}$,  $($where the superscript $T$ denotes the transpose$)$. Hence, denoting by $O_{n,n-1}$ the $n \times (n-1)$ zero matrix, one has 
\begin{align}\label{polarmatrix}
\widetilde{V}_{g}\zeta&=\begin{pmatrix}V_{g}\zeta_{1}&0\text{ }\cdots &0\\
                                             V_{g}\zeta_{2}&0\text{ }\cdots &0\\
                                             \vdots&\vdots\text{ }\cdots &\vdots\\                                           
                                             V_{g}\zeta_{n}&0\text{ }\cdots &0\\
                       \end{pmatrix}= \begin{pmatrix}
                                            (1\otimes v_{g}) \begin{pmatrix}
                                              \zeta_{1}\\\zeta_{2}\\\vdots\\\zeta_{n}
                                              \end{pmatrix}    O_{n,n-1}                                                            
                                           \end{pmatrix}\\&=(1\otimes v_{g})\begin{pmatrix}\zeta_{1}&0\text{ }\cdots &0\\ \nonumber
                                             \zeta_{2}&0\text{ }\cdots &0\\
                                             \vdots&\vdots\text{ }\cdots &\vdots\\                                           
                                             \zeta_{n}&0\text{ }\cdots &0\\
                       \end{pmatrix}\\
                      \nonumber &= (1\otimes v_{g})[u_{ij}][\eta_{ij}], \text{ }g\in G.
\end{align}

The partial isometry $w_g=(1\otimes v_{g})[u_{ij}]$ has the same initial space as $[u_{ij}]$, i.e.  $w_g^*w_{g}=[u_{ij}]^{*}[u_{ij}]$, hence Eq. \eqref{polarmatrix} implies that $L_{\widetilde{V}_g\zeta}=(1\otimes v_{g})[u_{ij}]L_{[\eta_{ij}]}$ is also the polar decomposition of $L_{\widetilde{V}_{g}\zeta}$ for all $g\in G$. Therefore,
by uniqueness of the polar decomposition it follows that for all $g \in G$, 
\begin{align}\label{Equal}
\widetilde{\beta}_{g}([u_{ij}])=(1\otimes v_{g})[u_{ij}] \text{ and } L_{\widetilde{V}_{g}[\eta_{ij}]}=L_{[\eta_{ij}]}.
\end{align}
Consequently, $\widetilde{V}_g[\eta_{ij}]=[\eta_{ij}]$ for all $g\in G$. 

First we investigate $[\eta_{ij}]$. Consider the decomposition $L_{[\eta_{ij}]}=\int_{0}^{\infty}\lambda de_{\lambda}$, where $e_{\lambda}\in \widetilde{M}=M\otimes M_{n}(\mathbb{C})$ for all $\lambda\geq 0$. Since $\widetilde{V}_{g}[\eta_{ij}]=[\eta_{ij}]$, arguing as in the proof of Theorem \ref{eigencase} we find $\widetilde{\beta}_{g}(e_{\lambda})=e_{\lambda}$ for all $\lambda\geq 0$. Thus, $e_{\lambda}\in 1\otimes M_{n}(\mathbb{C})$ for all $\lambda\geq 0$. Consequently, $[\eta_{ij}]=L_{[\eta_{ij}]}\Omega_{\widetilde{\varphi}}\in \mathbb{C}\Omega_{\varphi}\otimes M_{n}(\mathbb{C})$. 

Therefore,
\begin{align*}
\zeta=[u_{ij}][\eta_{ij}]\in (M\otimes M_{n}(\mathbb{C}))\Omega_{\widetilde{\varphi}},
\end{align*}
which implies that $\zeta_{j}\in M\Omega_{\varphi}$ for all $1\leq j\leq n$. This shows that $\mathcal{K}_{r}\subseteq M\Omega_{\varphi}$. Consequently, there is a finite dimensional $G$--invariant subspace $M_F$ of $M$ such that $\mathcal{K}_r=M_F\Omega_{\varphi}$ $($as $\Omega_{\varphi}$ is a separating vector of $M)$. Since every bounded subset of a finite--dimensional normed space is pre--compact with respect to its norm, the orbit $G\cdot y=\{\beta_g(y):g\in G\}$ of each $y\in M_F$ is a pre-compact subset of $M$. It is easy to see that $M_F$ is a subspace of the $C^\ast$-subalgebra 
\begin{align*}
M_\mathfrak{Cpt}=\{x\in M: G\cdot x \text{ is pre--compact in } \norm{\cdot}\}
\end{align*}
of $M$. By a celebrated result of 
H{\o}egh--Krohn,  Landstad, and St{\o}rmer \cite{HLS}, one has $M_\mathfrak{Cpt}\subseteq M^{\varphi}$, so $\mathcal{K}_r\subseteq M^{\varphi}\Omega_{\varphi}$ and consequently, $\mathcal{K}\subseteq L^{2}(M^{\varphi},\varphi)$.

Finally let 
\begin{align*}
\xi=\begin{pmatrix}\xi_{1}&0&0\text{ }\cdots &0\\
                                             \xi_{2}&0&0\text{ }\cdots &0\\
                                             \vdots&\vdots&\vdots\text{ }\cdots &\vdots\\                                           
                                             \xi_{m}&0&0\text{ }\cdots &0\\
                       \end{pmatrix}\in L^{2}(\widetilde{M}^{\widetilde{\varphi}},\widetilde{\varphi})\subseteq L^{2}(\widetilde{M},\widetilde{\varphi}).
\end{align*}
As a caution, note that to keep notation simple we have denoted both amplifications $M\otimes M_n(\mathbb{C})$ and $M\otimes M_m(\mathbb{C})$ by $\widetilde{M}$. However, this will cause no confusion.

Now, working in the setup of a finite von Neumann algebra equipped with a faithful normal tracial state, we can write the polar decompositon of $L_{\xi}$ as $L_{\xi}=[w_{ij}]L_{[\delta_{ij}]}$, where $[w_{ij}]\in \widetilde{M}^{\widetilde{\varphi}}$ and $L_{[\delta_{ij}]}$ is positive, self--adjoint and affiliated to $\widetilde{M}^{\widetilde{\varphi}}\otimes M_{m}(\mathbb{C})$. Using ergodicity, the same argument as above with $\zeta$ replaced by $\xi$ shows that $\delta_{ij}\in \mathbb{C}\Omega_{\varphi}$ for all $1\leq i,j\leq m$. Consequently, 
$\xi_{j}\in M^{\varphi}\Omega_{\varphi}$, which completes the proof.
\end{proof}

\begin{remark}
 \emph{An anonymous referee kindly provided the following sketch of an alternative approach to the above result, which we reproduce here with his or her permission. In the setting of Theorem \ref{InvaraintSubspaceInsideCentralizer} assume first that the finite--dimensional invariant subspace $\mathcal{K}$ is in the domain of $S_{\varphi}$. Let $\{\zeta_{i}\}_{i=1}^{n}$ be an orthonormal basis in $\mathcal{K}$. Consider the operator $$T=L_{\xi_{1}}^{\ast}L_{\xi_1}+\cdots+L_{\xi_{n}}^{\ast}L_{\xi_n},$$
more precisely, the positive operator defined by the closure of the quadratic form on $M'\Omega_{\varphi}$ defined by the above sum. Then $T$ is independent of the choice of $\{\zeta_{i}\}_{i=1}^{n}$ and is $G$--invariant. Hence it is scalar, so the operators $L_{\zeta_{i}}$ are bounded and $\mathcal{K}\subset M\Omega_{\varphi}$.
Consider the C$^{\ast}$ algebra $A$ generated by all finite--dimensional $G$--invariant subspaces of $M$. Letting $\widetilde{G}$ be the Bohr compactification of $G$ we get an action of $\widetilde{G}$ on $A$. Hence, by \cite{HLS}, the restriction of $\varphi$ to $A$ is tracial, so for the von Neumann algebra $N=A''$ we get that the modular group of $\varphi|_{N}$ is trivial. But as $N$ is $\sigma_{t}^{\varphi}$--invariant, we have $\sigma_{t}^{\varphi|_{N}}=\sigma_{t}^{\varphi}|_{N}$. Hence $N\subset M^{\varphi}$ and $\mathcal{K}\subset A\Omega_{\varphi}\subset M_{\varphi}\Omega_{\varphi}$. To obtain the general case, one approximates $\mathcal{K}$ by spaces of analytic vectors and apply the above argument to these spaces to obtain that $\mathcal{K}\subset \overline{M^{\varphi}\Omega_{\varphi}}$. one then applies the argument in the first paragraph once more, now to $M^{\varphi}$ instead of $M$, to finally obtain that $\mathcal{K}\subset M^{\varphi}\Omega_{\varphi}$.} 
\end{remark}

With Theorems \ref{eigencase} and \ref{InvaraintSubspaceInsideCentralizer} in hand, we will next finish the characterization of weakly mixing systems. After doing so, we will return to discuss some important consequences of Theorem \ref{InvaraintSubspaceInsideCentralizer}. 

\subsection{Characterization of Weak Mixing}

The main results of this section characterize weakly mixing systems in terms of joinings, 
providing a generalization to the noncommutative setting of an important basic result in classical ergodic theory. We emphasize that these results, unlike their recent analogues in \cite{Du2} and \cite{Du3}, do not require the assumption that the acting group is abelian. Furthermore, Theorem \ref{OtherwayWmDisjointCompact} establishes the existence of a nontrivial joining between an ergodic, non--weakly mixing system and an ergodic, compact system arising from an action on a type I von Neumann algebra. The methods we use to obtain this section's results appear to be necessary: the work in 
\cite{Du2} relies on a definition of weakly mixing actions that is not available when the acting group 
is nonabelian, and it is pointed out in \cite{Du3} that the techniques from harmonic analysis used 
there are particular to the abelian setting. Another potential approach to the results in \cite{Du3} 
follows an argument of Neshveyev and St{\o}rmer $($see Theorem 2.1 \cite{NeSt}$)$, however, 
these methods also do not generalize in an obvious way to the setting of a nonabelian acting group. 

We will require the following auxiliary result, which is an averaging technique implicit in \cite{Ta} and \cite {KoSz}. This is surely well--known to experts, but we provide a proof for the sake of completeness.

\begin{proposition}\label{average}$($Kov\'{a}cs--Sz\"{u}cs--Takesaki$)$
Let $\beta$ be a continuous action of a separable locally compact group $ G $ on a von Neumann 
algebra $M$ that preserves a faithful normal state $\varphi$. For $x\in M$, if 
$K(x)=\overline{conv}^{w^{*}}\{\beta_{g}(x):g\in  G \}$ 
then $K(x)\cap M^{ G }=\{\mathbb{E}_{M^{ G }}(x)\}$, where $\mathbb{E}_{M^{ G }}$ is the unique faithful normal conditional expectation \footnote{Since $\beta_{g}\sigma_{t}^{\varphi}=\sigma_{t}^{\varphi}\beta_{g}$
for all $t\in \mathbb{R}$ and $g\in  G $, by the main theorem of \cite{Ta} $($cf. \S3 \cite{Ta}$)$ we are guaranteed to have such a conditional expectation $\mathbb{E}_{M^{ G }}$.} of $M$ onto the fixed point algebra $M^{G}$ such that $\varphi \circ \mathbb{E}_{M^{ G }} =\varphi$.
\end{proposition}
 
\begin{proof}
Since $\varphi\circ\beta_{g}=\varphi$ for all $g\in  G $, 
$M$ is $ G $--finite \footnote{The collection of normal states invariant with respect to the $G$-action separates the positives of $M$.} $($see Proposition 1 and Defn. 1 of \cite{KoSz} for details$)$. 
Thus, by Theorem $1$ of \cite{KoSz} it follows that $K(x)\cap M^{ G }$
is a singleton set $\{\Phi(x)\}$. Furthermore, the map $x\mapsto \Phi(x)$ is a normal $ G $--equivariant projection of norm one by Theorem $2$ of \cite{KoSz}. By \cite{Ta} $($see pp. 310--311$)$, it follows that $\Phi(x)=\mathbb{E}_{M^{ G }}(x)$.
\end{proof}

The following lemma guarantees that a finite--dimensional invariant subspace associated to a trace--preserving group action on a finite von Neumann algebra induces a compact subsystem.  

\begin{lemma}\label{compactness}
Let $M$ be a finite von Neumann algebra in standard form with respect to a fixed faithful 
normal tracial state $\tau$. Let $G$ be a separable locally compact group with action $\gamma: g\mapsto  \gamma _{g}\in Aut(M,\tau)$. 
If $F\subseteq M$ is a finite--dimensional subspace such that $\gamma_{g}(F)\subseteq F$ for all $g \in G$ then the restriction of the action of $G$ to $N=vN(F \cup \{1\})$ defines a compact subsystem of $\mathfrak{M}=(M,\tau, \gamma, G)$.
\end{lemma}
 
\begin{proof}
Note that for each $x\in F$ the closure of the orbit $\{\gamma _{g}(x)\Omega_{\tau}\}_{g\in  G }$ is a precompact subset of $L^{2}(M,\tau)$  
$($where $\Omega_{\tau}$ is the standard vacuum vector$)$. 
Let $F^{*}=\{x^{*}: x\in F\}$. Then $\mathcal{A}_{0}=\text{span }(F\cup F^{*})$ is a self--adjoint 
finite--dimensional  subspace of $M$ which is $ G $--invariant. Let $\mathcal{A}$ be the 
$\ast$--algebra generated by $\mathcal{A}_{0}$. Note that for $x\in F$, and $g_{n} \in G$  
\begin{align*}
\gamma_{g _{n}}(x)\Omega_{\tau}\underset{n \rightarrow \infty}{\xrightarrow{\norm{\cdot}_{\tau}}} \xi\in L^{2}(M,\tau) \Leftrightarrow  \gamma_{g _{n}}(x^{*})\Omega_{\tau}\underset{n \rightarrow \infty}{\xrightarrow{\norm{\cdot}_{\tau}}} J_{\tau}\xi\in L^{2}(M,\tau).
\end{align*}
Passing to subseqences, we see that the orbit $\{ \gamma _{g}(x)\Omega_{\tau}\}_{g\in  G }$ of each 
$x\in \mathcal{A}_{0}$ is a precompact subset of $L^{2}(M,\tau)$. Clearly, 
$\mathcal{A}$ is $ G $--invariant, and so is $N$. 

The closed unit ball of $M$ $($and hence any norm--closed, bounded subset of $M$) is closed in $\norm{\cdot}_{\tau}$. Thus, if $x\in M$ and $\{\gamma_{g _{n}}(x)\}_n$ is Cauchy in $\norm{\cdot}_{\tau}$, then the limit in $\norm{\cdot}_{\tau}$ of $ \{\gamma_{g _{n}}(x)\Omega_{\tau}\}_{n}$ lies in $M\Omega_{\tau}$, and an easy triangle inequality argument shows that if $\gamma_{g _{n}}(x)\rightarrow x_{0}$ and $ \gamma_{g _{n}}(y)\rightarrow y_{0}$ in $\norm{\cdot}_{\tau}$, then 
$\gamma_{g _{n}}(xy)\rightarrow x_{0}y_{0}$ in $\norm{\cdot}_{\tau}$ for all $x,y\in M$. A routine argument involving passing to subsequences finitely many times establishes that for each $x\in\mathcal{A}$ the orbit $\{ \gamma _{g}(x)\Omega_{\tau}\}_{g\in  G }$ is a sequentially precompact subset of $L^{2}(M,\tau)$ and hence a precompact subset of $L^{2}(M,\tau)$. The precompactness of $\{\gamma _{g}(y)\Omega_{\tau}\}_{g\in  G }$ for $y\in N$ follows from the Kaplansky density theorem and a standard diagonalization argument.\end{proof}

Lemma \ref{compactness} highlights a  strong relationship between weakly mixing and compact systems arising from the action of a separable, locally compact group. If $\mathfrak{M}$ is an ergodic system which is not weakly mixing, then by the remark following Definition \ref{WM}, the representation $\pi_{V}$ of $G$ on $L^{2}(M,\varphi)$ admits a finite--dimensional invariant subspace $\mathcal{K} \subseteq L^{2}(M,\varphi) \ominus \mathbb{C}\Omega_\varphi$. It follows from Theorem \ref{InvaraintSubspaceInsideCentralizer} that there is a finite--dimensional invariant subspace $F \subseteq M$ such that $\mathcal{K}=F\Omega_{\varphi}$ and $\beta_{g}(F)\subseteq F$ for all $g\in G$. Moreover, $F\subseteq M^{\varphi}$, so Lemma \ref{compactness} applies, and we see that the action of $G$ restricted to the von Neumann algebra $N=vN(F\cup \{1\})\subseteq M^{\varphi}$ defines a nontrivial compact subsystem of the original system $\mathfrak{M}$.  These observations complete the proof of the following result. 

\begin{theorem}\label{CharacterizeWeakMixing} For an ergodic $W^{*}$--dynamical system, the following are equivalent:
\begin{enumerate}
\item The system is weakly mixing.
\item The system admits no nontrivial compact subsystem.
\item The system is disjoint from every compact system.
\end{enumerate}\end{theorem}

In the following result, we obtain further information about the compact systems with which an ergodic, non-weakly mixing system admits a nontrivial joining.  In particular, the underlying von Neumann algebra of at least one such compact system is either a type I factor or is abelian. 

\begin{theorem}\label{OtherwayWmDisjointCompact}
Let $\mathfrak{M}$ be an ergodic system such that the representation $\pi_{V}$ has a nonzero finite--dimensional reducing subspace orthogonal to $\Omega_{\varphi}$. Then 
there is an ergodic compact system $\mathfrak{N}$ whose underlying von Neumann algebra $N$ is finite--dimensional or abelian, and a joining $\Psi\in J_{m}(\mathfrak{M},\mathfrak{N})$ such that $\Psi\neq \varphi(\cdot)1_{N}$. 
\end{theorem}
 
\begin{proof}
The proof has three steps. In the first step, a suitable ergodic compact system is constructed from the initial data pertaining to the finite dimensional reducing subspace of the representation $\pi_{V}$. In the next step, 
a natural joining is established between $\mathfrak{M}$ and the compact system constructed in the first step. The joining obtained could be trivial, but in the final step the triviality of this joining allows us to obtain a new non trivial joining of $\mathfrak{M}$ with a suitable ergodic compact system.\\
\noindent \underline{\textbf{Step 1}}: 
Let $\mathcal{K}\subseteq L^{2}(M,\varphi)\ominus \mathbb{C}\Omega_{\varphi}$ 
be a nonzero finite--dimensional reducing subspace for $\pi_{V}$.  Let $P_\mathcal{K}: L^{2}(M,\varphi)\rightarrow \mathcal{K}$ 
denote the orthogonal projection. Then $G\ni g\overset{\pi_{\mathcal{K}}}\mapsto V_{g}P_{\mathcal{K}}\in\mathcal{U}(\mathbf{B}(\mathcal{K}))$ is a finite--dimensional representation of $G$ which thus decomposes into a direct sum of irreducible representations. Replacing $\mathcal{K}$ by one of the Hilbert subspaces that appear in this direct sum, we may, without loss of generality, assume that 
$\pi_{\mathcal{K}}$ is irreducible and hence that every nonzero vector in $\mathcal{K}$ is cyclic for $\pi_{\mathcal{K}}$. Let $\xi\in \mathcal{K}$ be a unit vector, and denote by $K$ the compact subgroup $\overline{\{V_{g}P_\mathcal{K}:g\in  G \}}^{\norm{\cdot}}$  of the unitary group of  $\textbf{B}(\mathcal{K})$.  Note that $K$ acts continuously on $\textbf{B}(\mathcal{K})$ by conjugation, i.e. $\alpha_{k}(y)=kyk^{*}$ for all $y\in \textbf{B}(\mathcal{K})$ and $k\in K$.
Taking $N=vN(\{k:k\in K\})=\textbf{B}(\mathcal{K})$, we see that $G$ acts on $N$ by $g\mapsto Ad(V_{g}P_{\mathcal{K}})$,  $g\in G$, and consequently, the action of $G$ on $N$ is ergodic. Let $\lambda$ denote the normalized Haar measure of $K$. Then the 
functional 
\begin{align}\label{Tracial}
\varrho(y)=\int_{K}\langle kyk^{*}\xi,\xi\rangle_{\varphi}d\lambda(k), \text{ }y\in N,
\end{align}
is invariant with respect to  $\alpha_{k}$ for all $k\in K$, since $\lambda$ is translation 
invariant. Thus,  $\varrho$ is the faithful normal tracial state on $N$ (Note that $\varrho$ need not be the vector state induced by $\xi$, and in fact is independent of the choice of $\xi$). 
Let $\mathfrak{N}=(N,\varrho,\alpha, G )$. By construction, $\mathfrak{N}$ is an ergodic, 
compact system. The associated action on the GNS representation will be denoted by 
$\alpha$, as well. This completes the first step. 

\noindent \underline{\textbf{Step 2}}: Let $\pi_{\varrho}$ denote the GNS representation of $N$ with 
respect to the tracial state $\varrho$. Thus, 
\begin{align}\label{TP3}
\varrho(y)=\int_{K}\langle kyk^{*}\xi,\xi\rangle_{\varphi}d\lambda(k)=
\langle \pi_{\varrho}(y) \Omega_{\varrho},\Omega_{\varrho}\rangle_{\varrho}, \text{ for all }y\in N, 
\end{align}
where $\Omega_{\varrho}$ denotes the distinguished cyclic vector in the GNS representation. 
Note that as $\varrho$ is tracial, $\sigma_{t}^{\varrho}=id$ for all $t\in\mathbb{R}$.  
Given $x\in M$, note that 
$\sigma_{t}^{\varphi}(x)=x$ for all $t\in \mathbb{R}$ 
if and only if $x\in M^{\varphi}$. Let $\mathbb{E}_{M^{\varphi}}$ denote the unique 
$\varphi$--preserving normal conditional expectation from $M$ onto $M^{\varphi}$, whose existence 
is guaranteed by \S3 \cite{Ta}. Define
\begin{align}\label{ConstructJoin}
\Psi: M\rightarrow N \text{ by }\Psi(x)=\pi_{\varrho}(P_{\mathcal{K}}\mathbb{E}_{M^{\varphi}}(x) P_{\mathcal{K}}), \text{ }x\in M.
\end{align}

Clearly, $\Psi$ is linear, u.c.p. and normal.  We now show that $\Psi$ has the desired properties. 
Note that for each $t\in\mathbb{R}$,
\begin{align*}
\Psi(\sigma_{t}^{\varphi}(x))&= \pi_{\varrho}(P_{\mathcal{K}}\mathbb{E}_{M^{\varphi}}(\sigma_{t}^{\varphi}(x)) P_{\mathcal{K}})\\
&=\pi_{\varrho}(P_{\mathcal{K}}\sigma_{t}^{\varphi}(\mathbb{E}_{M^{\varphi}}(x)) P_{\mathcal{K}}) \text{ (by the construction in \S3, \S4 \cite{Ta})}\\
&=\pi_{\varrho}(P_{\mathcal{K}}\mathbb{E}_{M^{\varphi}}(x) P_{\mathcal{K}})\\
&=\Psi(x)\\
&=\sigma_{t}^{\varrho}(\Psi(x)), \text{ }x\in M.
\end{align*}

Note that $\beta_{g}=Ad(V_{g})$ for $g\in  G $. By the uniqueness of $\mathbb{E}_{M^\varphi}$, we have 
$\beta_{g}\circ\mathbb{E}_{M^{\varphi}}=\mathbb{E}_{M^{\varphi}}\circ\beta_{g}$ for all 
$g\in  G $. Thus, for $x\in M$ and $g\in  G $, 
\begin{align*}
\Psi(\beta_{g}(x))&=\pi_{\varrho}(P_{\mathcal{K}}\mathbb{E}_{M^{\varphi}}(\beta_{g}(x)) P_{\mathcal{K}})\\
&=\pi_{\varrho}(P_{\mathcal{K}}\beta_{g}(\mathbb{E}_{M^{\varphi}}(x)) P_{\mathcal{K}})\\
&=\pi_{\varrho}(P_{\mathcal{K}}V_{g}\mathbb{E}_{M^{\varphi}}(x) V_{g}^{*}P_{\mathcal{K}})\\
&=\pi_{\varrho}(V_{g}P_{\mathcal{K}}\mathbb{E}_{M^{\varphi}}(x) P_{\mathcal{K}}V_{g}^{*}) \text{ (since }V_{g}P_{\mathcal{K}}=P_{\mathcal{K}}V_{g})\\
&=\pi_{\varrho}(V_{g}P_{\mathcal{K}})\pi_{\varrho}(P_{\mathcal{K}}\mathbb{E}_{M^{\varphi}}(x) P_{\mathcal{K}})\pi_{\varrho}(V_{g}P_{\mathcal{K}})^{*} \text{(as }\pi_{\varrho} \text{ is a }\ast\text{--homo.)}\\
&=\pi_{\varrho}(V_{g}P_{\mathcal{K}})\Psi(x)\pi_{\varrho}(V_{g}P_{\mathcal{K}})^{*}\\
&=\alpha_{g}(\Psi(x)).
\end{align*}

As $\Psi$ is $ G $--equivariant, for $x\in M$ we have for all $g\in G$ that
\begin{align*}
\varrho(\Psi(x))=\varrho(\alpha_{g}(\Psi(x)))=\varrho(\Psi(\beta_{g}(x))).
\end{align*}
Convex subsets of von Neumann algebras have the same strong--operator and weak--operator closures, so by Proposition \ref{average} we have that for $x\in M$, there is a sequence  $\{y_{n}\}_{n}$ in $conv\{\beta_{g}(x):g\in  G \}$ such that $y_{n}\rightarrow \mathbb{E}_{M^{ G }}(x)$ in $s.o.t.$ $($and in $\norm{\cdot}_{\varphi}$, as well$)$ as $n\rightarrow \infty$. Thus, 
\begin{align}\label{Tracezerounitary}
\varrho(\Psi(x))&=\varrho(\Psi(\mathbb{E}_{M^{ G }}(x))).\\
\nonumber&=\varphi(x) \text{ (by ergodicity)}\\
\nonumber&=\langle x\Omega_{\varphi},\Omega_{\varphi}\rangle_{\varphi}, \text{ for all }x\in M.
\end{align}
Thus, $\varrho\circ\Psi =\varphi$, so that $\Psi\in J_{m}(\mathfrak{M},\mathfrak{N})$.

\noindent \underline{\textbf{Step 3}}: Note that $\Psi$ could be trivial. In this case, as we demonstrate below, we can construct an abelian subalgebra $A$ of $M$ which is invariant with respect to the actions of both $\beta$ and the modular automorphism group of $\varphi$ and such that the action of $G$ on $A$ is compact. To finish the proof of the theorem, we will need to replace $\Psi$ by the joining $\mathbb{E}_{A}$ when the former is trivial. 

From Eq. \eqref{ConstructJoin} note that 
$\Psi(\cdot)=\varphi(\cdot)1_{N}$ if and only if 
$P_{\mathcal{K}}M^{\varphi}P_{\mathcal{K}}=\mathbb{C}P_{\mathcal{K}}$. This is obviously the case 
when $dim(\mathcal{K})=1$. So we argue first for the case when $dim(\mathcal{K})>1$. 

Assume that $\Psi$ is trivial. Then, $P_{\mathcal{K}}\mathbb{E}_{M^{\varphi}}(x)P_{\mathcal{K}}=\varphi(x)P_{\mathcal{K}}$, for all 
$x\in M$, as $\varrho\circ\Psi=\varphi$. Let $P: L^{2}(M,\varphi)\rightarrow L^{2}(M^{\varphi},\varphi)$ denote the orthogonal projection. 
Then, $\mathbb{E}_{M^{\varphi}}(x)P=PxP$ for all $x\in M$, by \S3 \cite{Ta}. Again by Theorem \ref{InvaraintSubspaceInsideCentralizer},
we have $P_{\mathcal{K}}\leq P$. Since 
$\varrho\circ\Psi =\varphi$ and $N=\mathbf{B}(\mathcal{K})$ and $P_{\mathcal{K}}\xi=\xi$, by Eq. \eqref{TP3} and \eqref{Tracezerounitary} for $x\in M$ we have
\begin{align}\label{GNS}
\langle\mathbb{E}_{M^{\varphi}}(x) \Omega_{\varphi},\Omega_{\varphi}\rangle_{\varphi}&=\langle x\Omega_{\varphi},\Omega_{\varphi}\rangle_{\varphi}\\ \nonumber &=\varphi(x)\\
&\nonumber =\varrho(\Psi(x))\nonumber\\&=\varrho(P_{\mathcal{K}}\mathbb{E}_{M^{\varphi}}(x)P_{\mathcal{K}})
=\langle\mathbb{E}_{M^{\varphi}}(x)\xi,\xi\rangle_{\varphi}. \nonumber
\end{align}
By Theorem \ref{InvaraintSubspaceInsideCentralizer}, $\xi\in M^{\varphi}\Omega_{\varphi}$ and we have $\langle \xi,\Omega_{\varphi}\rangle_{\varphi}=0$;  it follows that 
$M^{\varphi}$ is nontrivial. Consequently, 
\begin{align*}
L^{2}(M^{\varphi},\varphi)\ni x\Omega_{\varphi}\mapsto x\xi \in L^{2}(M^{\varphi},\varphi)
\end{align*}
implements an isometry $U: L^{2}(M^{\varphi},\varphi)\rightarrow L^{2}(M^{\varphi},\varphi)$ such that $$UxU^{*}=xUU^{*}$$ for all $x\in M^{\varphi}$ and 
$UU^{*}\in (M^{\varphi})^{\prime}\subseteq\mathbf{B}(L^{2}(M^{\varphi},\varphi))$, i.e. $UU^{*}$ projects onto 
$\overline{M^{\varphi}\xi}^{\norm{\cdot}_{\varphi}}$.  Thus, 
$U\in (M^{\varphi})^{\prime}\subseteq \textbf{B}(L^{2}(M^{\varphi},\varphi))$. 
But $\textbf{B}(L^{2}(M^{\varphi},\varphi))\supseteq (M^{\varphi})^{\prime}= \tilde{J}_{\varphi}M^{\varphi}\tilde{J}_{\varphi}$, 
where $\tilde{J}_{\varphi}$ denotes the restriction to the subspace $L^{2}(M^{\varphi},\varphi)$ of Tomita's anti--unitary $J_{\varphi}$ 
associated to $(M,\varphi)$. Thus, by the  
fundamental theorem of Tomita--Takesaki theory there exists an isometry $u\in M^{\varphi}$ such that 
$U=\tilde{J}_{\varphi}u^{*}\tilde{J}_{\varphi}$.  Since $(M^{\varphi},\varphi)$ is a tracial 
$W^{\ast}$--probability space,
\begin{align}\label{TraceZeroUnitary}
\xi =U\Omega_{\varphi}=\tilde{J}_{\varphi}u^{*}\tilde{J}_{\varphi}\Omega_{\varphi}=u\Omega_{\varphi}.
\end{align}
However, Eq. \eqref{GNS} and traciality then yield
\begin{align*}
\varphi(x)=\varphi(u^{*}xu)=\varphi(xuu^{*}), \text{ for all }x\in M^{\varphi}. 
\end{align*}
Since $\varphi$ is faithful, it follows that $uu^{*}=1$, whence $u\in M^{\varphi}$ is a trace--zero unitary. Since the construction of the joining $\Psi$ in Step 2 is independent of the choice of $\xi\in \mathcal{K}$ $($see Eq. \eqref{Tracial}, \eqref{TP3}$)$ and $\Psi$ is assumed trivial, it follows that \textit{every} unit vector in $\mathcal{K}$ is obtained via a trace zero unitary in $M^{\varphi}$ as in Eq. \eqref{TraceZeroUnitary}. We will show that this is absurd when $dim(\mathcal{K})>1$. 

Indeed, suppose $\xi_{1}=u_{1}\Omega_{\varphi}$ and $\xi_{2}=u_{2}\Omega_{\varphi}$ are orthogonal vectors in $\mathcal{K}$, with $u_{1},u_{2}$ unitaries in $M^{\varphi}$. Then $\frac{\xi_{1}+\xi_{2}}{\sqrt{2}}=(\frac{u_{1}+u_{2}}{\sqrt{2}})\Omega_{\varphi}$, and since $\Omega_{\varphi}$ is separating for the action of $M^{\varphi}$, it follows that $\frac{u_{1}+u_{2}}{\sqrt{2}}$ must be unitary and therefore $u_{1}^{*}u_{2}=-u_{2}^{*}u_{1}$. Replacing $u_{1}$ by $iu_{1}$ in the above yields that $-iu_{1}^{*}u_{2}=-iu_{2}^{*}u_{1}$. The above two equations imply that $u_{1}^{*}u_{2}=0$, a contradiction.

Consequently, $dim(\mathcal{K})=1$. It follows from Theorem \ref{eigencase} $($or otherwise$)$ that $\xi=v\Omega_{\varphi}$,
where $v\in M^{\varphi}$ is trace--zero unitary and there exists a (continuous) character $\chi$ of $G$ such that $\beta_{g}(v)=\langle \chi,{g}\rangle v$ for every $g\in G$; so $M^{\varphi}$ is also nontrival in this case.

Let $A=vN(v)\subseteq M^{\varphi}\subseteq M$. Then by Lemma \ref{compactness} the action of $G$ on $A$ is compact, and $A$ is clearly $\beta$--invariant because $v$ is an eigenvector. 
Let $\mathbb{E}_{A}$ denote the 
$\varphi$--preserving normal conditional expectation from $M$ onto $A$. If the joining $\Psi$ obtained in Step 2 were nontrivial, then we already have obtained a nontrivial joining of $\mathfrak{M}$ with an ergodic compact system $\mathfrak{N}$ in Step 2. Otherwise, replace $\mathfrak{N}$ by the ergodic compact system $(A,\varphi_{|A}, \beta_{|A},G)$, and note that  $\mathbb{E}_{A}\in J_{m}(\mathfrak{M},\mathfrak{N})$ is nontrivial. This completes the proof.
\end{proof}


\subsection{Mixing}

The phenomenon of mixing can be characterized in terms of joinings in the classical 
case $($cf. Proposition 7.3 \cite{Gl}$)$. A proof of the analogous fact in the context of abelian groups acting on von Neumann algebras appears in \S5 of \cite{Du2}. But by considering the appropriate notion of convergence, the generalization of the characterization of mixing to the context of arbitrary separable locally compact groups acting on von Neumann algebras is natural and easy in the \emph{tracial} setup. Recall that $G\ni g\mapsto \alpha_{g}\in Aut(N,\rho)$ is \textbf{mixing} if the analogue $\pi_{U}: G\ni g\mapsto U_{g}\in \textbf{B}(L^{2}(N,\rho))$ of the Koopman representation is $c_{0}$ on the orthocomplement of $\mathbb{C}1_{N}$. 

\begin{theorem}\label{Mixing}
Let $\mathfrak{N}$ be a system. If $\alpha_{ g }(x)\xrightarrow{w.o.t.} \rho(x) 1_{N}$ for all $x \in N$, then the action $\alpha$ of $ G $ is mixing. The converse is true when $\rho$ is tracial. 
\end{theorem}

\begin{proof}
Note that $\alpha_{ g }\in CP(N,\rho)$ for each $ g \in  G $. 
If $\alpha_{ g }(x)\xrightarrow{w.o.t.} \rho(x) 1_{N}$ for all $x \in N$.
Then for each $x,y
\in N$,
\begin{align*}
\langle U_{g} x\Omega_{\rho}, y\Omega_{\rho}\rangle_{\rho}=\rho(y^{*}\alpha_{g}(x))=\langle \alpha_{g}(x)\Omega_{\rho}, y\Omega_{\rho}\rangle_{\rho}\rightarrow \langle x\Omega_{\rho},\Omega_{\rho}\rangle_{\rho} \langle \Omega_{\rho},y\Omega_{\rho}\rangle_{\rho}.
\end{align*}
as $g\rightarrow \infty$. Thus, the action is mixing. 

Conversely, suppose the action of $ G $ is mixing and $\rho$ is tracial. Then for $x,y,z\in N$ as $g\rightarrow \infty$ we have
\begin{align*}
\langle \alpha_{g}(x)y\Omega_{\rho}, z\Omega_{\rho}\rangle_{\rho}=\rho(z^{*}\alpha_{g}(x)y)=\rho(yz^{*}\alpha_{g}(x))&\rightarrow \langle x\Omega_{\rho},\Omega_{\rho}\rangle_{\rho} \langle \Omega_{\rho},zy^{*}\Omega_{\rho}\rangle_{\rho}\\
&=\rho(x)\langle y\Omega_{\rho}, z\Omega_{\rho}\rangle_{\rho}\\
&= \langle \rho(x)1_{N} y\Omega_{\rho}, z\Omega_{\rho}\rangle_{\rho}.
\end{align*}
It follows that $\alpha_{ g }(x)\xrightarrow{w.o.t.} \rho(x) 1_{N}$ for all $x \in N$. 
\end{proof}

\section{Some Consequences}

In this section, we provide some applications of joinings $($some of which are essentially corollaries of Theorems \ref{InvaraintSubspaceInsideCentralizer} and \ref{OtherwayWmDisjointCompact}$)$ that do not arise in the context of classical ergodic theory, but are more relevant in the context of $($properly infinite$)$ von Neumann algebras. In all of the results from \ref{NontrivialCentralizer}--\ref{AbelianNormalizer}, we will freely use the machinery and notation from the proofs of Theorems \ref{InvaraintSubspaceInsideCentralizer} and \ref{OtherwayWmDisjointCompact}. 

The so-called ergodic hierarchy is often regarded as relevant for explicating the nature of randomness in deterministic dynamical systems, even though physical systems are not, in general, ergodic. The point is that almost all Hamiltonian systems are non--integrable and therefore exhibit varying degrees of randomness in large regions of their phase space. Thus, in statistical physics, the ergodic hierarchy is useful for describing the degree of randomness in these regions.  In the above spirit, we work to locate parts of the standard Hilbert space where the representation $\pi_{V}$ is compact and weakly mixing. 

\begin{corollary}\label{NontrivialCentralizer}
Let $\mathfrak{M}$ be an ergodic system such that the representation $\pi_{V}$ has a nonzero finite--dimensional reducing subspace 
orthogonal to $\Omega_{\varphi}$. Then 
$M^{\varphi}\neq \mathbb{C}1_{M}$. 
\end{corollary}

\begin{proof}
If $M^{\varphi}= \mathbb{C}1_{M}$, then the $\Psi\in J_{m}(\mathfrak{M},\mathfrak{N})$ defined
in Step 2 of Theorem \ref{OtherwayWmDisjointCompact} would be trivial. In that case, proceeding to 
Step 3 of Theorem \ref{OtherwayWmDisjointCompact}, we would obtain a trace zero unitary in $M^{\varphi}$,
which is impossible.
\end{proof}

It is a classical theorem of Halmos that an ergodic automorphism of a compact abelian group is mixing. In fact, the spectral measure of the action is Lebesgue and the spectral multiplicity is infinite \cite{Hal}. Likewise, in the context of $W^{\ast}$--dynamics, an ergodic action of $\mathbb{R}$ on a von Neumann algebra preserving a KMS state is weakly mixing, i.e. zero is the only simple eigenvalue of the associated Liouvillean $($see Theorem 1.3 of \cite{JP}$)$. It can be proved, independently of the above result, that the action of the modular group associated to a faithful normal state is weakly mixing if and only if this action is ergodic. The next
corollary says more. Note that the hypothesis $M^{\varphi}=\mathbb{C}1_{M}$ implies that $M$ is necessarily a type $\rm{III}_{1}$ factor. 

\begin{corollary}\label{ErgodicimplyWeakmixing}
Let $M$ be a von Neumann algebra and $\varphi$ a fixed faithful normal state on $M$. 
Suppose $M^{\varphi}=\mathbb{C}1_{M}$, and let a separable locally compact group $ G $ act on $M$ by $\varphi$--preserving 
automorphisms. If the action is ergodic, then it is weakly mixing.
\end{corollary}

\begin{proof}
Suppose to the contrary the action of $ G $ is not weakly mixing. Then, there is a nontrivial 
finite--dimensional reducing subspace $\mathcal{K}$ of $\pi_{V}$. By Corollary  
\ref{NontrivialCentralizer}, $M^{\varphi}\neq \mathbb{C}1_{M}$, which is  
absurd. 
\end{proof}

\begin{remark}
\emph{Note that there exists a faithful normal state $\varphi$ on the hyperfinite $\rm{III}_{1}$ factor arising from a gauge--invariant generalized free state satisfying the KMS condition on the CAR algebra whose
centralizer is trivial, see Theorem 1 of \cite{HeTa} and its corollary for details. The action of $\mathbb{R}$ via
modular automorphisms in the construction is, in fact, mixing. Also note that the hyperfinite $\rm{III}_{1}$ factor  has a dense family of ergodic 
states (Proposition 3 \cite{Cl}). } 
\end{remark}

\begin{corollary}\label{EigenSystemFinite}
Let $\mathfrak{M}$ be an ergodic system and let $\mathfrak{N}$ be its eigensubsystem, i.e. $N$ is the von Neumann
subalgebra of $M$ generated by the eigenfunctions of the action. Then, $N$ is a finite von Neumann algebra and
$\varphi_{|N}$ is a faithful, normal trace state. 
\end{corollary}

\begin{proof}
This follows directly from Theorem \ref{eigencase}.
\end{proof}

The next result generalizes Corollary \ref{ErgodicimplyWeakmixing} and locates a large ``region of randomness'' of the dynamics. 

\begin{corollary}\label{LocateWeakMixing}
Let $\mathfrak{M}$ be an ergodic system. Then the representation $\pi_{V}$ of $ G $ restricted to 
$L^{2}(M,\varphi)\ominus L^{2}(M^{\varphi},\varphi)$ is weakly mixing. 
\end{corollary}



The next result of this section is of independent interest in understanding normalizers of masas coming from
crossed products. 

\begin{corollary}\label{AbelianNormalizer}
Let $\mathfrak{M}$ be an ergodic system with $ G $ a countable discrete abelian group. Then the 
normalizing algebra $N(L( G ))^{\prime\prime}\subset M\rtimes_{\beta}  G $ is contained
inside the centralizer $(M\rtimes_{\beta}  G )^{\varphi}$. 
\end{corollary}

\begin{proof}
By Theorem 2.1 of \cite{NeSt} and the remark following it, the normalizing algebra 
$N(L( G ))^{\prime\prime}$ of $L( G )$ 
is equal to  $M_{0}\rtimes_{\beta}  G $, where $M_{0}$ is the von Neumann subalgebra of 
$M$ spanned by the eigenfunctions of the action. By Corollary \ref{EigenSystemFinite}, $M_{0}$ 
is a finite von Neumann algebra and $\varphi_{|M_{0}}$ is trace. It follows that 
$N(L( G ))^{\prime\prime}\subseteq (M\rtimes_{\beta} G )^{\varphi}$.
\end{proof}


\end{document}